\documentclass[a4paper,11pt]{article}

\usepackage[margin=1in]{geometry}
\usepackage{setspace}
\usepackage{amsmath}
\usepackage{amssymb}
\usepackage{amsthm}
\usepackage{mathrsfs}
\usepackage{esint}
\usepackage{xcolor}
\usepackage{hyperref}

\newtheorem{theorem}{Theorem}[section]
\newtheorem{proposition}[theorem]{Proposition}
\newtheorem{corollary}[theorem]{Corollary}
\newtheorem{lemma}[theorem]{Lemma}
\newtheorem{definition}[theorem]{Definition}
\newtheorem{remark}[theorem]{Remark}

\hypersetup{
    hidelinks,
    pdftitle={On the Hessian Conjecture in Lorentzian Signature: Constant Pivots and Hesse Systems},
    pdfauthor={Hanwen Liu}
}

\begin{document}

\title{\texorpdfstring{\textbf{On the Hessian Conjecture in Lorentzian Signature: Constant Pivots and Hesse Systems}}{On the Hessian Conjecture in Lorentzian Signature: Constant Pivots and Hesse Systems}}
\author{Hanwen Liu\\
{\small Mathematics Institute, University of Warwick}\\
{\small \href{mailto:hanwen.liu@warwick.ac.uk}{\texttt{hanwen.liu@warwick.ac.uk}}}}
\date{}

\maketitle

\begin{abstract}
As a close relative of the Jacobian conjecture, the Hessian conjecture in dimension $n$ states that the local Legendre transform of a polynomial solution to the Monge--Amp\`ere equation $\det(\operatorname{Hess}(\phi))=\pm1$ is also a polynomial solution. The general Hessian conjecture is false for $n\geq5$, while in Riemannian signature it follows from the J\"orgens--Calabi--Pogorelov theorem.

We study the four-dimensional Hessian conjecture in Lorentzian signature. For a polynomial potential $\phi$ in four real variables whose Hessian matrix has index $1$ and determinant $-1$, we define a constant pivot for $\phi$ to be a nonzero constant vector $\xi$ such that the second directional derivative $D_\xi^2\phi$ is constant. We then prove that the gradient mapping of every potential admitting a pivot is a polynomial automorphism, and that a pivot always exists when $\phi$ decomposes into homogeneous pieces as $\phi=\phi_d+\phi_{d-1}+\phi_2+\phi_1+\phi_0$ with $d\geq4$. More generally, we prove the same conclusion when $$\phi=\phi_d+\cdots+\phi_{d-k}+\phi_2+\phi_1+\phi_0,$$ where $k\geq0$ and $d\geq4k+3$. Then, we associate with each potential a linear system of quadrics, called the Hesse system, and a canonical homomorphism $\mu_\phi$. We prove that the existence of a pivot is equivalent to $\operatorname{rank}(\mu_\phi)\leq55$. As an application, we prove that the gradient mapping is a polynomial automorphism whenever the Hesse system has complex dimension at most 4. We also show that if ${\det(\operatorname{Hess}(\phi-\phi_2))\equiv0}$, then the potential $\phi$ admits a pivot. After that, we then give an analytic degeneracy criterion for $\operatorname{Hess}(\phi-\phi_2)$. Finally, we prove the Hessian conjecture in this setting for every polynomial potential of degree at most five.
\end{abstract}

\begin{center}

\end{center}

\tableofcontents
\onehalfspacing
\raggedbottom

\section{Introduction}

\subsection{Historical Background and Overview}

For a polynomial function $\phi$, the Jacobian determinant of its gradient mapping is $\det(\operatorname{Hess}(\phi))$. The Hessian conjecture asks whether the gradient mapping is a polynomial automorphism whenever this determinant is a nonzero constant. In this article, we study the four-dimensional real case under the additional assumption that the Hessian has index $1$.

The modern formulation of the Hessian conjecture was introduced by Meng \cite{Meng2006}. It belongs to the symmetric-gradient branch of Keller's Jacobian conjecture \cite{Keller1939}. More precisely, the Jacobian conjecture in dimension $n$ implies the Hessian conjecture in dimension $n$, whereas the Hessian conjecture in dimension $2n$ implies the Jacobian conjecture in dimension $n$.
Thus, the two universal conjectures are equivalent when all dimensions are considered, but this equivalence is not dimension-preserving. The second implication is obtained by applying the Hessian conjecture to the doubled potential $\langle F(x),y\rangle$ associated with a polynomial Keller mapping $F$ \cite{Meng2006,deBondtEssen2005}.

The connection with the Jacobian conjecture is also reflected in its reduction theory. The degree reduction of Bass, Connell and Wright, followed by the symmetric reduction of de Bondt and van den Essen, reduces the all-dimensional problem, after increasing the dimension, to mappings of the form $x-\nabla P(x)$, where {the polynomial} $P$ is a homogeneous quartic polynomial and $\operatorname{Hess}(P)$ is nilpotent \cite{BassConnellWright1982,deBondtEssen2005}. Zhao subsequently reformulated this Hessian-nilpotent case as a vanishing conjecture for iterated Laplacians of powers of $P$ \cite{Zhao2007}. These reductions explain why constant-determinant Hessians form a central testing ground for the Jacobian problem.

This conjecture should not be confused with Hesse's older claim concerning homogeneous forms whose Hessian determinant vanishes identically. Gordan and Noether proved the cone conclusion for forms in at most four variables and constructed counterexamples from five variables onward; a modern account is given in \cite{BricalliFavalePirola2023}. The homogeneous Hesse theorem in the valid low-dimensional range is nevertheless an important tool in the present article, because the highest homogeneous part of a constant-Hessian potential has vanishing Hessian determinant.

The first low-dimensional results were obtained by Dillen in two variables and by de Bondt in three variables \cite{Dillen1991,deBondt2015}. Over the real numbers, Meng also proved the conjecture whenever the Hessian is positive or negative definite at one point \cite{Meng2006}. This agrees with the global rigidity of convex solutions to the constant Monge--Amp\`ere equation supplied by the theorems of J\"orgens, Calabi and Pogorelov \cite{Jorgens1954,Calabi1958,Pogorelov1972}. Consequently, Lorentzian signature is the first indefinite case and provides a natural next setting for the problem.

The status changed substantially in 2026. Building on Alp\"oge's three-variable counterexample to the Jacobian conjecture, a recent preprint of Meng and Yang constructs an explicit polynomial in five variables, of degree $14$ and Hessian determinant $128$, whose gradient mapping is not injective \cite{MengYang2026}. Stabilisation then gives counterexamples in every dimension at least five. Together with the preceding positive results, this leaves the Hessian conjecture open only in dimension four; similarly, the Jacobian conjecture is open only in dimension two, and the four-dimensional Hessian conjecture implies the two-dimensional Jacobian conjecture. The four-dimensional Lorentzian problem studied here therefore lies in the remaining dimension and includes, besides the quartic case, the two-layer and multi-layer classes described below.

The arguments are organised around one elementary device: a fixed direction in which the second derivative of the potential is constant. Such a direction will be called a constant pivot. The causal type of a pivot determines the corresponding inversion argument. A timelike pivot forces the potential to be quadratic, a lightlike pivot gives an explicit normal form, and a spacelike pivot reduces the problem to the known three-dimensional case.

The first part of the article develops two methods for producing constant pivots. The method of highest weight detects a pivot from a non-singular Newton face. A separate degree-by-degree argument proves the existence of a pivot for potentials having only two nonlinear homogeneous layers. After the Hesse-system criteria, this argument is extended to a consecutive block of $k+1$ nonlinear homogeneous layers whenever $d\geq4k+3$.

The second part introduces the Hesse system, which records all nonconstant Hessian differences as a linear system of quadrics. We obtain a finite complex rank criterion for the existence of a pivot and prove that, in Lorentzian signature, a complex pivot always produces a real one. This gives a directly computable criterion and, as an application, settles every potential whose Hesse system has complex dimension at most four.

The remaining sections treat singular Hesse systems and return to multi-layer potentials, followed by the singular-residual criterion. The main text concludes with the general quintic case, while the appendix records a null normal form for potentials admitting a pivot.

We finally record an equivalent Legendre-transform formulation of the Hessian conjecture. Let $\kappa:=\det(\operatorname{Hess}(\phi))$ be a nonzero constant. Around any point, write $p=\nabla\phi(x)$ and denote its local inverse by $x=x(p)$. The local Legendre transform is
$$
\phi^*(p):=\langle x(p),p\rangle-\phi(x(p)).
$$
Computation then yields $\nabla\phi^*(p)=x(p)$ and ${\operatorname{Hess}(\phi^*)_p=\operatorname{Hess}(\phi)_x^{-1}}$, so in particular it holds that $\det(\operatorname{Hess}(\phi^*))=\kappa^{-1}$. Thus, the Hessian conjecture is equivalently the assertion that the local Legendre transform of a polynomial solution to the Monge--Amp\`ere equation extends to a polynomial solution on the whole affine space. In the present Lorentzian setting, this is the local classical Legendre transform rather than the convex Fenchel transform.

The remainder of this section fixes the notation and records the restrictions imposed on the highest homogeneous part of the potential.

\subsection{Notations and Conventions}

Throughout this article, we use the following notation.

By abuse of notation, the symbol $\mathbb{R}^n$ in this manuscript will denote either the $n$-dimensional real vector space, or its underlying affine variety $\mathbb{A}^n$ over $\mathbb{R}$, depending on the exact context. We then denote by $\langle-,-\rangle$ the standard inner product on $\mathbb{R}^n$. {In dimension four, we put $V:=\mathbb{R}^4$, $V^*:=\operatorname{span}_{\mathbb{R}}\{x_1,\dots,x_4\}$ and $V_{\mathbb{C}}:=V\otimes_{\mathbb{R}}\mathbb{C}$.}

For any nonzero polynomial $\phi\in\mathbb{R}[x_1,\dots,x_n]$ of degree $d$, we write $\phi_k$ for the homogeneous piece of $\phi$ of degree $k$, so that {we have} $\phi=\phi_d+\cdots+\phi_0$ and $\phi_k(\lambda x)=\lambda^k\phi_k(x)$.

{We put}
$$
{\phi_{\geq3}:=\sum_{j\geq3}\phi_j=\phi-(\phi_2+\phi_1+\phi_0).}
$$

For any $\phi\in\mathbb{R}[x_1,\dots,x_n]$, we write $\nabla\phi$ for the gradient vector field of $\phi$ and $\operatorname{Hess}(\phi)$ for the Hessian matrix of $\phi$ in the coordinate system $x=(x_1,\dots,x_n)$.

For a tensor field $T$ on $\mathbb{R}^n$, as usual, its specific value at $x\in\mathbb{R}^n$ shall be denoted by $T_x$. For example, we shall write $H_\phi(x):=\operatorname{Hess}(\phi)_x$ and $H_{\phi_i}(x):=\operatorname{Hess}(\phi_i)_x$. We also write $G:=\operatorname{Hess}(\phi_2)$ and $R_\phi(x):=\operatorname{Hess}(\phi_{\geq3})_x$. When $G$ is non-singular, we write $L_\phi(x):=G^{-1}R_\phi(x)$.

One exception is that, in this article, we sometimes also regard the gradient $\nabla\phi$ of $\phi\in\mathbb{R}[x_1,\dots,x_n]$ as a regular map from $\mathbb{R}^n$ to $\mathbb{R}^n$ itself. In this special case, we will always explicitly call $\nabla\phi\colon\mathbb{R}^n\rightarrow\mathbb{R}^n$ a gradient mapping, and instead of the tensor notation $\nabla\phi_x$, the image of a point $x\in\mathbb{R}^n$ under the regular map $\nabla\phi$ will be denoted by $\nabla\phi(x)$.

\begin{definition}[Unimodular Hessian Potential]\label{unimodular_hessian_potential}
A polynomial function $\phi\in\mathbb{R}[x_1,\dots,x_n]$ is termed a unimodular Hessian potential of index $q$ if $\det(\operatorname{Hess}(\phi))=\pm1$ and {$\operatorname{ind}(\operatorname{Hess}(\phi))=q$}.
\end{definition}

\begin{definition}[Constant Pivot]\label{constant_pivot}
For any $\phi\in\mathbb{R}[x_1,\dots,x_n]$, a constant pivot for $\phi$ is a nonzero constant vector {$\xi\in\mathbb{R}^n$} for which {$\langle \operatorname{Hess}(\phi)\xi,\xi\rangle$} is a constant real number.
\end{definition}

For a nonzero constant vector {$\xi\in\mathbb{R}^n$}, we write {$D_\xi$} for the directional derivative along {$\xi$}. Then, equivalently, a constant pivot for $\phi\in\mathbb{R}[x_1,\dots,x_n]$ is a nonzero constant vector {$\xi\in\mathbb{R}^n$} such that {$D_\xi^2\phi$} is a constant.

The constant Hessian-determinant condition first imposes a restriction on the highest homogeneous piece of the potential.

\begin{lemma}\label{top_hessian_singularity}
Let $\phi\in\mathbb{R}[x_1,\dots,x_4]$ be a unimodular Hessian potential of index $1$ and degree $d\geq3$. Then, the matrix $\operatorname{Hess}(\phi_d)$ is singular.
\end{lemma}

\begin{proof}
The homogeneous component of degree $4(d-2)$ of $\det(\operatorname{Hess}(\phi))$ is {$\det(\operatorname{Hess}(\phi_d))$}.
Since $\det(\operatorname{Hess}(\phi))$ is constant, we obtain that this component vanishes.
\end{proof}

This singularity allows us to isolate the variables which genuinely occur in the highest homogeneous piece.

\begin{lemma}\label{essential_space}
Let $\phi\in\mathbb{R}[x_1,\dots,x_4]$ be a unimodular Hessian potential of index $1$ and degree $d\geq3$. Then, there exists a unique minimal linear subspace $${V_{\mathrm{ess}}(\phi_d)\subseteq V^*}$$ of dimension {$r_{\mathrm{ess}}(\phi_d)\in\{1,2,3\}$} such that for some polynomial $F$ in {$r_{\mathrm{ess}}(\phi_d)$} variables and some {ordered basis $u$} of {$V_{\mathrm{ess}}(\phi_d)$}, it holds that $F(u)=\phi_d(x)$.
\end{lemma}

\begin{proof}
{Let $K_d:=\{v\in V\mid D_v\phi_d=0\}$.} By Lemma~\ref{top_hessian_singularity} and the homogeneous Hesse theorem in four variables \cite{deBondtEssen2004}, there exists a nonzero complex vector $${\zeta=a+\sqrt{-1}b\in V_{\mathbb{C}}}$$ such that {$D_\zeta\phi_d=0$}. Since $\phi_d$ has real coefficients, we have that {$D_a\phi_d=0$}. We also have that {$D_b\phi_d=0$}.
Consequently, the vector space {$K_d$} is nonzero. Denote its annihilator space by $${V_{\mathrm{ess}}(\phi_d):=\operatorname{Ann}(K_d)\subseteq V^*.}$$

It is then readily seen that $\phi_d$ is a polynomial on {$V_{\mathrm{ess}}(\phi_d)$}. If $\phi_d$ is a polynomial on another linear subspace {$E\subseteq V^*$}, then we have that {$\operatorname{Ann}(E)\subseteq K_d$}, and hence we obtain that {$$V_{\mathrm{ess}}(\phi_d)\subseteq E.$$} Therefore, the vector space {$V_{\mathrm{ess}}(\phi_d)$} is unique and minimal. Since $\phi_d$ is nonzero and {$K_d$} is nonzero, we obtain that {$1\leq r_{\mathrm{ess}}(\phi_d)\leq3$}.
Moreover, homogeneity and $d\geq3$ show that {$K_d$} is precisely the common kernel of the matrices {$H_{\phi_d}(x)$} for $x\in\mathbb{R}^4$.
\end{proof}

\begin{definition}\label{essential_rank}
For a unimodular Hessian potential $\phi\in\mathbb{R}[x_1,\dots,x_4]$ of index $1$ and degree $d\geq3$, the unique minimal vector space {$V_{\mathrm{ess}}(\phi_d)$} in Lemma~\ref{essential_space} is termed the space of essential parameters of $\phi_d$, and its dimension {$r_{\mathrm{ess}}(\phi_d):=\dim_\mathbb{R}(V_{\mathrm{ess}}(\phi_d))$} is said to be the essential rank of $\phi_d$.
\end{definition}

The Lorentzian signature imposes a further restriction when the degree is odd.

\begin{lemma}\label{lorentz_parity}
Let $\phi\in\mathbb{R}[x_1,\dots,x_4]$ be a unimodular Hessian potential of index $1$ and degree $d\geq3$. Assume that $d$ is an odd integer. Then, the essential rank of $\phi_d$ is at most 2.
\end{lemma}

\begin{proof}
Now, fix some $x\in\mathbb{R}^4$.
For $t>0$, we have that
$$
t^{2-d}{H_\phi(tx)}\rightarrow {H_{\phi_d}(x)}
$$
as $t\rightarrow+\infty$. Since every matrix on the left has index $1$, the matrix {$H_{\phi_d}(x)$} has at most one negative eigenvalue. Applying the same argument at $-x$, we obtain that {$H_{\phi_d}(-x)$} has at most one negative eigenvalue. Since $d$ is odd, we have that {$H_{\phi_d}(-x)=-H_{\phi_d}(x)$}. Thus, the matrix {$H_{\phi_d}(x)$} also has at most one positive eigenvalue, and hence we obtain that {$\operatorname{rank}(H_{\phi_d}(x))\leq2$}.
If {$r_{\mathrm{ess}}(\phi_d)=3$}, then we would have that $\phi_d$ is an essential ternary form with singular Hessian. The ternary homogeneous Hesse theorem would reduce it to at most two essential parameters, which is a contradiction.
\end{proof}

Having isolated these highest-degree restrictions, we now explain why the existence of a constant pivot is sufficient for polynomial inversion.

\section{Constant Pivots}

The role of a constant pivot depends on its causal type. Timelike pivots are rigid, lightlike pivots admit an explicit normal form, and spacelike pivots reduce the problem to three variables.

\begin{lemma}\label{timelike_pivot_rigidity}
Let $\phi\in\mathbb{R}[x_1,\dots,x_4]$ be a unimodular Hessian potential of index $1$. Suppose that a constant vector {$\xi\neq0$} satisfies {$\langle\operatorname{Hess}(\phi)\xi,\xi\rangle=a<0$}.
Then, the polynomial $\phi$ is quadratic.
\end{lemma}

\begin{proof}
After a unimodular linear change of coordinates, write $x=(u,t)$ and {$\xi=\partial_t$}. Integrating twice with respect to $t$, we obtain that $$\phi(u,t)=\frac{a}{2}t^2+b(u)t+c(u).$$
Put $s:={\partial_t\phi}=at+b(u)$ and define
$$
L_s(u):=c(u)-\frac{(b(u)-s)^2}{2a}\in\mathbb{R}[u_1,u_2,u_3].
$$
The Schur-complement identities give that $L_s$ is strictly convex. We also have that
$$
\det(\operatorname{Hess}(L_s))=-\frac{1}{a}.
$$
For every fixed $s\in\mathbb{R}$, the J\"orgens--Calabi--Pogorelov theorem \cite{Jorgens1954,Calabi1958,Pogorelov1972} shows that $L_s$ is quadratic. Now, put $$Q(u):=c(u)-\frac{b(u)^2}{2a}.$$
We then have that
$$
L_s(u)=Q(u)+\frac{s}{a}b(u)-\frac{s^2}{2a}.
$$
Consequently, the polynomials $Q$ and $b$ have degree at most two. Since the Hessian pencil of $L_s$ is positive definite for every $s\in\mathbb{R}$, we obtain that $\operatorname{Hess}(b)=0$. Thus, the polynomial $b$ is affine and the polynomial $c$ is quadratic. Therefore, we conclude that $\phi$ is quadratic.
\end{proof}

We next consider the lightlike case. Although the potential need not be quadratic, it has a normal form from which the inverse of its gradient can be read off.

\begin{lemma}\label{lightlike_pivot_normal_form}
Let $\phi\in\mathbb{R}[x_1,\dots,x_4]$ be a unimodular Hessian potential of index $1$. Suppose that a constant vector {$\xi\neq0$} satisfies {$\langle\operatorname{Hess}(\phi)\xi,\xi\rangle=0$}. After a linear change of coordinates and addition of an affine polynomial, there exist variables $t,s$ and $y:=(y_1,y_2)$ such that
$$
\phi(t,s,y)=\frac{1}{2}\langle A(s)y,y\rangle+\langle\beta(s),y\rangle+c(s)+ts,
$$
where the entries of $A$ and $\beta$ and the function $c$ are all polynomial in the variable $s$, while the matrix $A(s)$ is positive definite and has determinant one.

Moreover, the gradient mapping $\nabla\phi$ is a polynomial automorphism.
\end{lemma}

\begin{proof}
Put {$\xi=\partial_t$}. Since {$\partial_t^2\phi=0$}, we may write $\phi(u,t)=a(u)t+b(u)$, where the vector $u$ belongs to $\mathbb{R}^3$. In the ordered coordinate system $(u,t)$, we have the block form
$$
\operatorname{Hess}(\phi)=
\begin{pmatrix}
\operatorname{Hess}(b)+\operatorname{Hess}(a)t&\nabla a\\
{\nabla a^\top}&0
\end{pmatrix},
$$
whose non-singularity gives $da\neq0$. Write {$K:=\ker(da_x)$}. The positive screen of the lightlike vector $\partial_t$ is
$$
{(\operatorname{Hess}(b)_x+t\operatorname{Hess}(a)_x)|_{K}.}
$$
Since this affine pencil is positive definite for every $t\in\mathbb{R}$, we obtain that {$\operatorname{Hess}(a)_x|_{K}=0$}.
The second fundamental form of every level surface $\Sigma$ of $a=a(u)$ therefore vanishes. Thus, the unit normal of $\Sigma$ is constant on each of its connected components $S_1,\dots,S_m$, so that each such component lies in an affine plane and is open therein. Since each component $S_i$ is also closed therein, it equals the plane. Distinct level planes are disjoint and hence parallel. Consequently, there exists a linear coordinate $s$ such that $a=f(s)$.
Writing $b_s(y):=b(s,y)$ for each $s\in\mathbb{R}$ and expanding the bordered determinant gives
$$
-f'(s)^2\det(\operatorname{Hess}(b_s))=\kappa
$$
for some constant $\kappa\in\mathbb{R}^\times$. Both factors are polynomials whose product is a nonzero constant. Thus, the polynomial $f'$ is a nonzero constant, and we have that $\det(\operatorname{Hess}(b_s))$ is a positive constant. After subtracting the affine summand $f(0)t$ and rescaling $s$, we may assume that $a(s)=s$. We then have that $\operatorname{Hess}(b_s)>0$. After a constant linear change of the screen variables, we may also assume that $\det(\operatorname{Hess}(b_s))=1$.
For every fixed $s\in\mathbb{R}$, J\"orgens' theorem \cite{Jorgens1954} shows that $b_s(y)$ is quadratic in $y$. Since $b=b(u)$ is a polynomial, we obtain that
$$
b(s,y)=\frac{1}{2}\langle A(s)y,y\rangle+\langle\beta(s),y\rangle+c(s).
$$
This proves the normal form.

Write {$p:=(\partial_1\phi,\partial_2\phi)$}{, where we put $\partial_j:=\partial/\partial y_j$}. We first recover $s$ from {$s=\partial_t\phi$}, and then recover $y$ from $y=A(s)^{-1}(p-\beta(s))$. Finally, we recover $t$ from
$$
t={\partial_s\phi}-\frac{1}{2}\langle A'(s)y,y\rangle-\langle\beta'(s),y\rangle-c'(s).
$$
Since $\det(A(s))=1$, the matrix $A(s)^{-1}$ is polynomial as it is then precisely the adjugate of $A(s)$. Hence the formulas above give a polynomial inverse of $\nabla\phi$.
\end{proof}

\subsection{Auxiliary Matrix Lemmas}

The arguments below also require two elementary consequences of Lorentzian signature and of the homogeneous Hesse theorem.

\begin{lemma}\label{lorentz_flat_pencil}
Let $A$ be a real symmetric matrix of signature $(3,1)$, and let $B$ be a real symmetric matrix. Suppose that $\det(A+sB)=\det(A)$ for every $s\in\mathbb{R}$. Then, we have that $\operatorname{rank}(B)\leq2$.
\end{lemma}

\begin{proof}
The pencil $A+sB$ is non-singular and has signature $(3,1)$ for every $s$. Dividing by $s>0$ and letting $s\rightarrow+\infty$, we obtain that $B$ has at most one negative eigenvalue. Now, replacing $s$ by $-s$, then dividing by $s>0$, and finally letting $s\rightarrow+\infty$, we obtain that $-B$ has at most one negative eigenvalue. Thus, the matrix $B$ has at most one positive eigenvalue. Therefore, we conclude that $\operatorname{rank}(B)\leq2$.
\end{proof}

\begin{lemma}\label{small_hessian_kernel}
Let {$P$} be a homogeneous polynomial of degree at least two in four variables over a field of characteristic zero. 
\begin{enumerate}
    \item [1.] If the rank of {$\operatorname{Hess}(P)$} at a generic point is at most 2, then we have that {$P$} depends on at most 2 linear forms. In particular, the matrix field {$\operatorname{Hess}(P)$} has a constant kernel of dimension at least 2.
    \item [2.] If the rank of {$\operatorname{Hess}(P)$} at a generic point is at most 1, then we have that {$P$} depends on at most 1 linear form. In particular, the matrix field {$\operatorname{Hess}(P)$} has a constant kernel of dimension at least 3.
\end{enumerate}
\end{lemma}

\begin{proof}
Apply the homogeneous Hesse theorem successively to the space of essential parameters. Whenever its dimension is larger than the generic Hessian rank, its Hessian determinant vanishes and one further inessential direction can be removed. The resulting constant kernel descends from the algebraic closure, since it is the solution space of a homogeneous linear system over the ground field. The proof is therefore completed.
\end{proof}

\subsection{Polynomial Inversion via Pivots}

Together with the two causal lemmas above, the three-variable Hessian theorem now settles every constant pivot.

\begin{proposition}[Constant-Pivot Theorem]\label{constant_pivot_inversion}
Let $\phi\in\mathbb{R}[x_1,\dots,x_4]$ be a unimodular Hessian potential of index $1$. If $\phi$ admits a constant pivot, then the gradient mapping $\nabla\phi\colon\mathbb{R}^4\rightarrow\mathbb{R}^4$ is a polynomial automorphism.
\end{proposition}

\begin{proof}
Let {$\xi$} be a constant pivot for $\phi$, and put {$a:=\langle\operatorname{Hess}(\phi)\xi,\xi\rangle$}.
If $a<0$, Lemma~\ref{timelike_pivot_rigidity} shows that $\phi$ is quadratic. If $a=0$, Lemma~\ref{lightlike_pivot_normal_form} already gives a polynomial inverse.

It remains to assume that $a>0$. After a unimodular linear change of coordinates, write $x=(u,t)$ and {$\xi=\partial_t$}. We have that $$\phi(u,t)=\frac{a}{2}t^2+b(u)t+c(u).$$
Put $s:={\partial_t\phi}=at+b(u)$ and define
$$
L_s(u):=c(u)-\frac{(b(u)-s)^2}{2a}.
$$
At a fixed $s\in\mathbb{R}$, the Schur-complement identities show that $\nabla L_s$ is the $u$-component of $\nabla\phi$ after the substitution $$t=\frac{s-b(u)}{a},$$ and also give $\det(\operatorname{Hess}(L_s))=-1/a$. By de Bondt's three-variable theorem \cite{deBondt2015}, applied over the rational function field $\mathbb{R}(s)$, the gradient mapping $\nabla L_s$ has a rational inverse. Hence the polynomial map that sends $(u,s)$ to $(\nabla L_s,s)$ is birational and has nonzero constant Jacobian determinant. By the birational Keller theorem \cite{Keller1939}, this mapping is in fact a polynomial automorphism. It then follows immediately that $\nabla\phi$ itself is also a polynomial automorphism.
\end{proof}

Thus, it remains to find a constant pivot. We shall do this first by highest weights and degree separation, and later by a complex linear system of quadrics.

\section{A Theorem on Two-Layer Potentials}

We now prove the main degree-theoretic result. In degrees at least six, dilation separates the two nonconstant Hessian layers. The quintic and quartic cases require additional arguments according to the essential rank.

We first record a convexity lemma which will be used in the quintic rank-one case.

\begin{lemma}\label{convex_cylinder}
Let $f\in\mathbb{R}[x_1,\dots,x_n]$ be a convex polynomial function such that $\det(\operatorname{Hess}(f))=0$. Then, there exists a nonzero constant vector $v$ such that $\operatorname{Hess}(f)v=0$.
\end{lemma}

\begin{proof}
If $f$ were strictly convex, then we would have that $\nabla f$ is injective. Invariance of domain would make its image open, whereas Sard's theorem would make the same image of Lebesgue measure zero. Thus, the function $f$ has to be affine on a segment parallel to a vector $v\neq0$. Its restriction to the containing line is a polynomial and is therefore affine on the whole line. We write $$f(x_0+tv)=f(x_0)+ct.$$
For arbitrary $x$, the supporting hyperplane inequality gives
$$
f(x_0+tv)\geq f(x)+\langle\nabla f(x),x_0+tv-x\rangle.
$$
Letting $t\rightarrow+\infty$ and then $t\rightarrow-\infty$, we obtain that $\langle\nabla f(x),v\rangle=c$.
Differentiation gives the assertion.
\end{proof}

\subsection{The Method of Highest Weight}

The highest-weight method detects a lightlike pivot whenever a weighted leading part retains nonzero Hessian determinant. We begin with the elementary weight calculation behind this observation.

\begin{lemma}\label{weighted_hessian_face}
Assign to $x_1,\dots,x_4$ {the positive integral weight tuple $\omega:=(\omega_1,\dots,\omega_4)$}. Suppose that {$\operatorname{in}_{\omega}(\phi)$} has weight $N$. If {$\det(\operatorname{Hess}(\operatorname{in}_{\omega}(\phi)))\neq0$}, then {this determinant} is the highest weight component of $\det(\operatorname{Hess}(\phi))$ and has weight {$4N-2(\omega_1+\omega_2+\omega_3+\omega_4)$}.
\end{lemma}

\begin{proof}
The $(i,j)$-th entry of {$\operatorname{Hess}(\operatorname{in}_{\omega}(\phi))$} has weight {$N-\omega_i-\omega_j$}. Therefore, every term in {$\det(\operatorname{Hess}(\operatorname{in}_{\omega}(\phi)))$} has weight {$4N-2(\omega_1+\omega_2+\omega_3+\omega_4)$}, while every term in $\det(\operatorname{Hess}(\phi))$ involving a lower weight component of $\phi$ has smaller weight. The desired statement follows.
\end{proof}

We now express the equality case in terms of the Newton polytope.

For a polynomial $\phi\in\mathbb{R}[x_1,\dots,x_n]$, we denote its support, namely the set of its exponent vectors, by $\operatorname{supp}(\phi)$, and the convex hull of its support by $\operatorname{Newt}(\phi)$. Traditionally, the convex body $\operatorname{Newt}(\phi)$ is nothing but the Newton polytope of $\phi\in\mathbb{R}[x_1,\dots,x_n]$. Given positive weights {$\omega=(\omega_1,\dots,\omega_n)$}, we also denote by {$\operatorname{in}_{\omega}(\phi)$} the component of $\phi\in\mathbb{R}[x_1,\dots,x_n]$ having the highest weight.

\begin{proposition}\label{newton_polytope_criterion}
Let $\phi\in\mathbb{R}[x_1,\dots,x_4]$ be a polynomial function with constant Hessian determinant $\det(\operatorname{Hess}(\phi))=\kappa\in\mathbb{R}^\times$. Let {$\operatorname{in}_{\omega}(\phi)$} have weight $N$ for some positive weights {$\omega=(\omega_1,\dots,\omega_4)$}, and put {$c:=(1/2,\dots,1/2)$}. Then, it holds that {$c\in \operatorname{Newt}(\phi)$}.
Moreover, the following conditions are equivalent:
\begin{enumerate}
    \item [1.] {$\det(\operatorname{Hess}(\operatorname{in}_{\omega}(\phi)))\neq0$}.
    \item [2.] {$2N=\omega_1+\omega_2+\omega_3+\omega_4$}.
    \item [3.] The exposed face {$F:=\{\alpha\in \operatorname{Newt}(\phi)\mid\langle \omega,\alpha\rangle=N\}$} contains {$c$}.
\end{enumerate}
The above conditions also imply {$\det(\operatorname{Hess}(\operatorname{in}_{\omega}(\phi)))=\kappa$}. Equivalently, after normalising $N=1$, the existence of the vector {$c$} in the face $F$ is precisely the feasibility of positive numbers {$\omega_1,\dots,\omega_4$} satisfying
$$
{\frac{1}{2}\sum_{i=1}^4\omega_i=1,}
$$
and {$\langle \omega,\alpha\rangle\leq1$} for every $\alpha\in\operatorname{supp}(\phi)$.
\end{proposition}

\begin{proof}
Denote by $(e_1,\dots,e_4)$ the standard orthonormal basis of $\mathbb{R}^4$, and write $P:=\operatorname{Newt}(\phi)$. Since $\det(G)=\kappa\neq0$, some permutation term in $\det(G)$ is nonzero. Choose $\sigma\in\mathfrak{S}_4$ such that $$\prod_{i=1}^4G_{\sigma(i)i}\neq0.$$ Then, the monomial $x_ix_{\sigma(i)}$ occurs in $\phi_2$ for every $i$, and
$$
{c=\frac{1}{4}\sum_{i=1}^4e_i+\frac{1}{4}\sum_{i=1}^4e_{\sigma(i)}\in P.}
$$
Thus, we have that
$$
{N\geq\langle \omega,c\rangle=\frac{1}{2}\sum_{i=1}^4\omega_i.}
$$
Put {$m:=2N-(\omega_1+\omega_2+\omega_3+\omega_4)$} and, for {$t>0$}, put
$$
{\phi^t(x):=t^N\phi(t^{-\omega_1}x_1,\dots,t^{-\omega_4}x_4).}
$$
Then, the polynomial {$\phi^t$} converges to {$\operatorname{in}_{\omega}(\phi)$} as {$t\rightarrow0$}, and we have that
$$
{\det(\operatorname{Hess}(\phi^t))
=\kappa t^{2m}.}
$$
The preceding inequality shows that the exponent $m$ is nonnegative. Passing to the limit as {$t\rightarrow0$}, we obtain that {$\det(\operatorname{Hess}(\operatorname{in}_{\omega}(\phi)))\neq0$} if and only if $m=0$, and in this case it equals $\kappa$. The vanishing of the exponent is equivalent to {$N=\langle \omega,c\rangle$}, which is also equivalent to the exposed face $F$ containing {$c$}. 

The final formulation then follows by normalising $N=1$.
\end{proof}

When the equivalent conditions above hold, one variable has weight greater than half the weight of the potential. The corresponding pure second derivative must therefore vanish.

\begin{theorem}\label{nonsingular_face_pivot}
Let $\phi\in\mathbb{R}[x_1,\dots,x_4]$ be a nonquadratic polynomial such that $\det(\operatorname{Hess}(\phi))\in\mathbb{R}^\times$. Suppose that
$$
{\det(\operatorname{Hess}(\operatorname{in}_{\omega}(\phi)))\neq0}
$$
for some positive weights {$\omega=(\omega_1,\dots,\omega_4)$}. Then some vector {$\xi\in\mathbb{R}^4$} satisfies {$D_\xi^2\phi=0$}.
\end{theorem}

\begin{proof}
Let $N$ be the highest weight of $\phi$. Proposition~\ref{newton_polytope_criterion} gives
$$
{\sum_{i=1}^4\omega_i=2N.}
$$
Choose $j$ so that {$\omega_j=\max\{\omega_1,\omega_2,\omega_3,\omega_4\}$}. Then, we have that {$\omega_j\geq N/2$}. Equality would force every weight to equal $N/2$, and hence would force $\phi$ to be quadratic. Therefore, it must be the case that {$\omega_j>N/2$}. Notice that every monomial component of {$\partial_j^2\phi$} has weight at most {$N-2\omega_j<0$}, which gives {$\partial_j^2\phi=0$}, as all weights are assumed to be positive. Thus, the vector {$\xi=\partial_j$}{, where we put $\partial_j:=\partial/\partial x_j$,} is a constant lightlike pivot.
\end{proof}

This criterion produces a pivot from a single non-singular Newton face. The next section treats two-layer potentials for which no such face need be chosen.

\subsection{The High-Degree Case}

We retain the notation $H_{\phi_k}(x)$ introduced above.

With this notation, dilation separates the two nonconstant Hessian layers, except for one explicitly controlled collision.

\begin{lemma}\label{dilation_separation}
Let $\phi\in\mathbb{R}[x_1,\dots,x_4]$ be a unimodular Hessian potential of index $1$ and degree $d\geq6$. Suppose that
$$
\phi=\phi_d+\phi_{d-1}+\phi_2+\phi_1+\phi_0.
$$
Let {$r_{\mathrm{ess}}(\phi_d)$} be the essential rank of $\phi_d$. If $d\geq7$, or if $d=6$ and {$r_{\mathrm{ess}}(\phi_d)\leq2$}, then 
$$
\det({H_{\phi_2}}+s{H_{\phi_{d-1}}(x)}+t{H_{\phi_d}(x)})=\det({H_{\phi_2}})
$$
for every $x\in\mathbb{R}^4$ and $s,t\in\mathbb{R}$.
\end{lemma}

\begin{proof}
For every fixed $x\in\mathbb{R}^4$, the substitution of $x$ by $\lambda x$ gives
$$
\det({H_{\phi_2}}+\lambda^{d-3}{H_{\phi_{d-1}}(x)}+\lambda^{d-2}{H_{\phi_d}(x)})=\det({H_{\phi_2}}).
$$
The coefficient of $s^it^j$ in $\det({H_{\phi_2}}+s{H_{\phi_{d-1}}(x)}+t{H_{\phi_d}(x)})$ is multiplied by $\lambda^{i(d-3)+j(d-2)}$, where the integers $i,j$ are nonnegative and have sum at most four.

Suppose first that $d\geq7$. The map sending $(i,j)$ to $i(d-3)+j(d-2)$ is injective on the set $I:=\{(i,j)\in\mathbb{Z}^2:i,j\geq0,i+j\leq4\}$, and hence every nonconstant coefficient of $s^it^j$ in
$$
\det({H_{\phi_2}}+s{H_{\phi_{d-1}}(x)}+t{H_{\phi_d}(x)})
$$
must vanish.

Suppose that $d=6$ and {$r_{\mathrm{ess}}(\phi_d)\leq2$}. The map sending $(i,j)$ to $i(d-3)+j(d-2)$ fails to be one-to-one only over $12\in\mathbb{Z}$, with the collision arising from the coefficients of $s^4$ and $t^3$, while the latter coefficient vanishes because {$\operatorname{rank}(H_{\phi_6})\leq2$}, and hence the former coefficient also has to vanish.

The desired identity then follows.
\end{proof}

We now introduce another suggestive notation: {For a matrix field $T$ on $\mathbb{R}^n$,} the vector space 
$$
{\mathcal{K}(T):=\bigcap_{x\in\mathbb{R}^n}\ker(T_x)}
$$
is termed the common kernel of {$T_x$} for $x\in\mathbb{R}^n$.

We shall combine the preceding determinant identity with the common kernels of the two highest Hessian layers.

\begin{lemma}\label{high_degree_pivot}
Let $\phi\in\mathbb{R}[x_1,\dots,x_4]$ be a unimodular Hessian potential of index $1$ and degree $d\geq6$. Suppose that
$$
\phi=\phi_d+\phi_{d-1}+\phi_2+\phi_1+\phi_0.
$$
Then, the polynomial function $\phi$ admits a constant pivot.
\end{lemma}

\begin{proof}
{Put $r:=r_{\mathrm{ess}}(\phi_d)$.} Also, let {$K_d:=\mathcal{K}(H_{\phi_d})$}, and choose a complementary direction space {$E_d$}. We claim that the restriction {$H_{\phi_d}|_{E_d}$} is generically non-singular. Indeed, otherwise the homogeneous Hesse theorem on the space of essential parameters would reduce its dimension. 

We also have that {$H_{\phi_2}=H_\phi(0)$}, and hence the symmetric matrix {$H_{\phi_2}$} has signature $(3,1)$.

Suppose first that $d=6$ and {$r=3$}. Let {$\xi\in K_6$} be a spanning vector of {$K_6$}. The coefficient of $\lambda^{15}$ in the polynomial
$$
\det({H_{\phi_2}}+\lambda^3{H_{\phi_5}(x)}+\lambda^4{H_{\phi_6}(x)})
$$
is a nonzero constant multiple of {$\det(H_{\phi_6}(x)|_{E_6})\langle H_{\phi_5}(x)\xi,\xi\rangle$}. Since the factor {$\det(H_{\phi_6}(x)|_{E_6})$} is generically nonzero, we obtain that {$\langle H_{\phi_5}(x)\xi,\xi\rangle=0$}. Since we also have that {$H_{\phi_6}(x)\xi=0$}, the vector {$\xi$} is a constant pivot.

It remains to consider the cases covered by Lemma~\ref{dilation_separation}. In these cases, Lemma~\ref{lorentz_flat_pencil} shows that the ranks of {$H_{\phi_d}(x)$} and {$H_{\phi_{d-1}}(x)$} are at most two. In particular, Lemma~\ref{small_hessian_kernel} gives {$r\leq2$}.

Let {$K_{d-1}:=\mathcal{K}(H_{\phi_{d-1}})$}. Then, Lemma~\ref{small_hessian_kernel} gives {$\dim(K_{d-1})\geq2$}. The coefficient of {$s^{4-r}t^r$} in the determinant identity of Lemma~\ref{dilation_separation} is a nonzero constant multiple of
$$
{\det(H_{\phi_d}(x)|_{E_d})\det(H_{\phi_{d-1}}(x)|_{K_d})}.
$$
Consequently, the restricted operator {$H_{\phi_{d-1}}|_{K_d}$} is singular.

If {$r=1$}, dimension counting gives {$K_{d-1}\cap K_d\neq0$}. Suppose now that {$r=2$}. If the rank of {$H_{\phi_{d-1}}$} at a generic point is at most one, then Lemma~\ref{small_hessian_kernel} gives {$\dim(K_{d-1})\geq3$}, and hence again we obtain that {$K_{d-1}\cap K_d\neq0$}. Otherwise, if {$K_{d-1}\cap K_d=0$}, then the natural homomorphism from {$K_d$} to {$\mathbb{R}^4/K_{d-1}$} is an isomorphism, and hence the form induced by {$H_{\phi_{d-1}}$} on {$\mathbb{R}^4/K_{d-1}$} is generically non-singular, but under the preceding isomorphism, its pullback is the restriction {$H_{\phi_{d-1}}|_{K_d}$}, which is a contradiction.

So, we may now choose a nonzero vector {$\xi\in K_{d-1}\cap K_d$}. Then, we have {$H_{\phi_d}(x)\xi=0$}. As we also have {$H_{\phi_{d-1}}(x)\xi=0$}, it follows that {$D_\xi^2\phi=\langle H_{\phi_2}\xi,\xi\rangle$} is a constant.
\end{proof}

This proves the existence of a pivot in every degree at least six. For quintics, Lemma~\ref{lorentz_parity} leaves only essential ranks one and two.

\subsection{Two-Layer Quintic Potentials}

The argument below concerns the two-layer quintic case. We return to the general quintic case near the end of the main text.

The rank-two case reduces to a positive semi-definite singular 2 by 2 matrix of quadratic forms. We first record the required dichotomy.

\begin{lemma}\label{quadratic_matrix_dichotomy}
Let $C=C(x)$ be a 2 by 2 symmetric matrix whose entries are homogeneous quadratic polynomials in $n$ variables with real coefficients. Suppose that $C$ is positive semi-definite at every real point and that $\det(C)=0$. Then, either the common kernel of $C(x)$ for $x\in\mathbb{R}^n$ is non-trivial, or the matrix $C$ has the form {$C=\rho\,\ell\otimes\ell$}, where the real constant $\rho$ is positive and the two components of the {$\ell$} are independent linear forms.
\end{lemma}

\begin{proof}
For later use, we will write
$$
C:=
\begin{pmatrix}
a&b\\
b&c
\end{pmatrix}.
$$
If $a=0$ or $c=0$, positive semi-definiteness gives $b=0$, and hence the common kernel of $C(x)$ for $x\in\mathbb{R}^n$ is non-trivial for trivial reasons. 

Otherwise, the unique factorization of $ac=b^2$ gives three homogeneous polynomials $\lambda,p,q$ such that $(a,b,c)=(\lambda p^2,\lambda pq,\lambda q^2)$.
Homogeneity now shows that either the polynomial $\lambda$ is quadratic and the polynomials $p,q$ are constant, or the polynomial $\lambda$ is a positive constant and the polynomials $p,q$ are linear. The first case above gives a common kernel. Linear dependence of $p$ and $q$ also gives a common kernel, while the remaining case gives the asserted expression.
\end{proof}

\begin{lemma}\label{quintic_rank_two_pivot}
Let $\phi\in\mathbb{R}[x_1,\dots,x_4]$ be a unimodular Hessian potential of index $1$ of the form
$$
\phi=\phi_5+\phi_4+\phi_2+\phi_1+\phi_0.
$$
If {$r_{\mathrm{ess}}(\phi_5)=2$}, then the polynomial function $\phi$ admits a constant pivot.
\end{lemma}

\begin{proof}
Let {$K_5:=\mathcal{K}(H_{\phi_5})$}, and take an ordered linear coordinate system {$x=(u,z)$} adapted to a direct sum decomposition {$V=E_5\oplus K_5$}, where the vectors {$u,z$} belong to $\mathbb{R}^2$. We also set {$C:=H_{\phi_4}|_{K_5}$}.

The proof of Lemma~\ref{lorentz_parity} shows that {$H_{\phi_5}$} has at most one eigenvalue of either sign. Since the rank of {$H_{\phi_5}$} at a generic point is two, the restriction {$H_{\phi_5}|_{E_5}$} has signature $(1,1)$ at a generic point. Along the ray $\{\lambda x:\lambda\geq0\}$, the Schur complement on {$K_5$} is asymptotically $\lambda^2C+O(\lambda)$. Since the restriction {$H_{\phi_5}|_{E_5}$} already has index one, we immediately obtain that $C\geq0$.

The coefficient of $\lambda^{10}$ in
$$
\det({H_{\phi_2}}+\lambda^2{H_{\phi_4}(x)}+\lambda^3{H_{\phi_5}(x)})
$$
is a nonzero constant multiple of {$\det(H_{\phi_5}(x)|_{E_5})\det(C(x))$}. Therefore, we obtain that $\det(C)=0$.

By Lemma~\ref{quadratic_matrix_dichotomy}, either {$\mathcal{K}(C)$} is non-trivial or {$C=\rho\,\ell\otimes\ell$}, where the constant $\rho$ is positive and the two components of the {$\ell$} are independent linear forms. In the first case, any choice of a nonzero vector in {$\mathcal{K}(C)$} is a constant pivot. 

Suppose that the second case above occurs. The Hessian integrability identities in the variables {$z_1$} and {$z_2$}, together with the independence of the two components {$\ell_1$} and {$\ell_2$} of {$\ell$}, readily show that {$\ell$} is independent of {$z$}. After a linear change of variables, we may therefore assume that {$\ell\equiv u$}. Two integrations with respect to {$z=(z_1,z_2)$} then give that the expression $${\phi_4-\frac{\rho}{2}\langle u,z\rangle^2}$$ is affine in {$z$}.

Now, we give the variables $u_1,u_2$ weight two and the variables {$z_1,z_2$} weight three. {Thus, for $\omega:=(2,2,3,3)$, the highest weight component of $\phi$ is}
$$
{\operatorname{in}_{\omega}(\phi)=\phi_5(u)+\frac{\rho}{2}\langle u,z\rangle^2,}
$$
because, by homogeneity, the remaining part of $\phi_4$ which is linear in {$z$} has weight nine, while the component independent of {$z$} has weight eight. Every term in $\phi$ of degree at most 2 has weight at most 6.
A direct block matrix computation gives
$$
{\det(\operatorname{Hess}(\operatorname{in}_{\omega}(\phi)))}
={\rho^3\langle u,z\rangle^2(3\rho\langle u,z\rangle^2-\langle H_{\phi_5}(u)u,u\rangle).}
$$
Euler's identity on homogeneous functions yields {$\langle H_{\phi_5}(u)u,u\rangle=20\phi_5(u)$}. Thus, the above determinant {$\det(\operatorname{Hess}(\operatorname{in}_{\omega}(\phi)))$} is nonzero. But Lemma~\ref{weighted_hessian_face} now contradicts the constancy of $\det(\operatorname{Hess}(\phi))$. Therefore, the second case actually cannot occur.
\end{proof}

We now treat essential rank one. Here the transverse quartic Hessian is positive semi-definite and singular, so Lemma~\ref{convex_cylinder} supplies the required direction.

\begin{lemma}\label{quintic_rank_one_pivot}
Let $\phi\in\mathbb{R}[x_1,\dots,x_4]$ be a unimodular Hessian potential of index $1$ of the form
$$
\phi=\phi_5+\phi_4+\phi_2+\phi_1+\phi_0.
$$
If {$r_{\mathrm{ess}}(\phi_5)=1$}, then the polynomial function $\phi$ admits a constant pivot.
\end{lemma}

\begin{proof}
By the hypothesis, we may choose a leading variable $t$ and complementary variables $z\in\mathbb{R}^3$, so that {we have} $\phi_5=\alpha t^5$ for some $\alpha\in\mathbb{R}^\times$. For every fixed $t\in\mathbb{R}$, denote by $C(t,z)$ the Hessian matrix of the polynomial function $\phi_4(t,-)$ in the variables $z$.

For each point with $t\neq0$, we choose the sign of $\lambda$ such that $20\alpha\lambda^3t^3<0$, and let $|\lambda|\rightarrow\infty$. After division of the transverse Schur complement by $\lambda^2$ and passage to the limit $|\lambda|\rightarrow\infty$, we obtain that $C(t,z)\geq0$. Continuity then gives the same conclusion when $t=0$.

The coefficient of $\lambda^9$ in
$$
\det({H_{\phi_2}}+\lambda^2{H_{\phi_4}(t,z)}+\lambda^3{H_{\phi_5}(t,z)})
$$
has one contribution containing one entry from {$H_{\phi_5}$} and three entries from {$H_{\phi_4}$}, whose value is $20\alpha t^3\det(C(t,z))$, whilst its only other contribution contains three entries from the rank-one matrix {$H_{\phi_5}$} and therefore vanishes. Hence, we obtain that $\det(C)=0$.

Lemma~\ref{convex_cylinder}, applied to the polynomial function $\phi_4(1,-)$, gives a nonzero constant vector {$\xi\in\mathbb{R}^3$} such that {$C(1,z)\xi=0$}. Homogeneity yields $C(t,z)=t^2C(1,z/t)$ for $t\neq0$, and continuity gives {$C(0,z)\xi=0$}. Regarding {$\xi$} as a transverse vector in $\mathbb{R}^4$, we also have that {$H_{\phi_5}(t,z)\xi=0$}. It is readily seen that {$\langle H_{\phi_4}(t,z)\xi,\xi\rangle=0$}. 

Thus, the vector {$\xi$} is a constant pivot for $\phi$.
\end{proof}

\subsection{Quartic Potentials}

It remains to consider quartic potentials. We first record the elementary structure of a singular pencil of symmetric 2 by 2 matrices.

\begin{lemma}\label{singular_binary_pencil}
Let $C=C(x)$ be a 2 by 2 symmetric matrix whose entries are $\mathbb{R}$-linear forms in $n$ variables. If $\det(C(x))=0$ for every $x$, then the common kernel of $C(x)$ for $x\in\mathbb{R}^n$ is non-trivial.
\end{lemma}

\begin{proof}
If $C$ vanishes identically, the assertion is immediate. Otherwise, after a linear change of basis in $\mathbb{R}^2$, one nonzero matrix in the linear space $\{C(x):x\in\mathbb{R}^n\}$ is $\operatorname{diag}(a,0)$ for some $a\neq0$. For every other matrix $$
D=\begin{pmatrix}
D_{11} & D_{12}\\
D_{21} & D_{22}
\end{pmatrix}
$$ in $\{C(x):x\in\mathbb{R}^n\}$, we have that $\det(\operatorname{diag}(a,0)+sD)=0$ for every $s\in\mathbb{R}$, from which we obtain the equalities $D_{22}=D_{12}=0$. Thus, the common kernel of $C(x)$ for $x\in\mathbb{R}^n$ contains the second standard basis vector $e_2\in\mathbb{R}^2$.
\end{proof}

We now apply this pencil lemma to the restriction of the cubic Hessian to the kernel of the quartic Hessian.

\begin{lemma}\label{quartic_pivot}
Let $\phi\in\mathbb{R}[x_1,\dots,x_4]$ be a unimodular Hessian potential of index $1$ of the form
$$
\phi=\phi_4+\phi_3+\phi_2+\phi_1+\phi_0.
$$
Then, the polynomial function $\phi$ admits a constant pivot.
\end{lemma}

\begin{proof}
{Put $r:=r_{\mathrm{ess}}(\phi_4)$.} Also, let {$K_4:=\mathcal{K}(H_{\phi_4})$}, and choose a complementary direction space {$E_4$}, so that {we have $V=K_4\oplus E_4$}. As before, the restriction {$H_{\phi_4}|_{E_4}$} is non-singular at a generic point. The contribution containing {$r$} entries from {$H_{\phi_4}$} and {$4-r$} entries from {$H_{\phi_3}$} {to}
$$
{[\lambda^{4+r}]\det(H_{\phi_2}+\lambda H_{\phi_3}(x)+\lambda^2H_{\phi_4}(x))}
$$
is a nonzero constant multiple of {$\det(H_{\phi_4}(x)|_{E_4})\det(H_{\phi_3}(x)|_{K_4})$}.
Assume that {$r\leq2$}. Then, the only other possible contribution to {the coefficient above}
contains {$r+1$} entries from {$H_{\phi_4}$}, and hence it vanishes because {$\operatorname{rank}(H_{\phi_4})\leq r$}.
Consequently, the restriction of {$H_{\phi_3}$} to {$K_4$} is singular.

Suppose now that {$r=3$}. Then, we have that {$K_4=\mathbb{R}\xi$} for some nonzero vector {$\xi$}, and comparing with {the coefficient extracted above} gives {$\langle H_{\phi_3}(x)\xi,\xi\rangle=0$}. Since {$H_{\phi_4}(x)\xi=0$} as well, the vector {$\xi$} is a constant pivot.

Suppose next that {$r=2$}. The restriction of {$H_{\phi_3}$} to the two-dimensional space {$K_4$} is a linear pencil of singular symmetric matrices. Lemma~\ref{singular_binary_pencil} immediately gives a constant pivot.

It remains to assume that {$r=1$}. Choose a leading variable $t$ and complementary variables $z\in\mathbb{R}^3$. Write
$$
{\phi_3=\beta t^3+P_1(z)t^2+P_2(z)t+P_3(z),}
$$
where the polynomials {$P_1,P_2,P_3$} are homogeneous of degrees 1,2,3, respectively. We now put {$B:=\operatorname{Hess}(P_2)$} and {$C:=\operatorname{Hess}(P_3)$}. The singularity of the restriction of {$H_{\phi_3}$} to {$K_4$} shows $\det(tB+C(z))=0$. Investigating the coefficient of $t^3$ then gives $\det(B)=0$.

It now suffices to consider the following three subcases:

Suppose first that $\operatorname{rank}(B)=2$, and let {$\xi$} span {$\ker(B)$}. The coefficient of $t^2$ is a nonzero constant multiple of {$\langle C(z)\xi,\xi\rangle$}, and hence this expression vanishes.

Suppose next that $B$ has rank one. The coefficient of $t$ is a nonzero constant multiple of the determinant of the restriction of $C(z)$ to {$\ker(B)$}, and hence this restricted operator is a linear pencil of singular symmetric matrices. Lemma~\ref{singular_binary_pencil} then gives a nonzero vector {$\xi\in\ker(B)$} such that {$\langle C(z)\xi,\xi\rangle=0$}.

Finally, if $B=0$, the ternary homogeneous Hesse theorem immediately gives a nonzero vector in the common kernel of $C(z)$ for $z\in\mathbb{R}^3$. 

Therefore, we have proved that, in every case, there exists a nonzero vector {$\xi\in K_4$} such that {$\langle(tB+C(z))\xi,\xi\rangle=0$}.

This vector {$\xi$} is a constant pivot for $\phi$.
\end{proof}

The preceding lemmas cover all possible degrees and essential ranks, and therefore assemble into the following statement.

\begin{theorem}[Two-Layer Pivot]\label{two_layer_pivot}
Let $\phi\in\mathbb{R}[x_1,\dots,x_4]$ be a unimodular Hessian potential of index $1$ and degree $d\geq4$. Suppose that $\phi$ is of the form
$$
\phi=\phi_d+\phi_{d-1}+\phi_2+\phi_1+\phi_0.
$$
Then, the polynomial function $\phi$ admits a constant pivot.
\end{theorem}

\begin{proof}
If $d\geq6$, the assertion follows from Lemma~\ref{high_degree_pivot}. Suppose that $d=5$. Lemma~\ref{lorentz_parity} gives {$r_{\mathrm{ess}}(\phi_5)\leq2$}, and hence Lemma~\ref{quintic_rank_two_pivot} or Lemma~\ref{quintic_rank_one_pivot} applies. Finally, if $d=4$, the assertion follows from Lemma~\ref{quartic_pivot}. 

The desired theorem is therefore settled.
\end{proof}

Combining this theorem with Proposition~\ref{constant_pivot_inversion} proves the Hessian conjecture for the stated two-layer potentials. We now turn to a degree-free criterion for the existence of a pivot.

\section{Complexifications and the CR Criterion}

The purpose of this section is to detect constant pivots by linear algebra over $\mathbb{C}$ and then return to the real Lorentzian setting. The complex criterion is algebraic; the main point is that every complex pivot forces a real one.

\subsection{The Complex Pivot Criterion}

Throughout this section, we complexify $\mathbb{R}^n$ to $\mathbb{C}^n=\mathbb{R}^n\otimes\mathbb{C}$, and denote by $\langle-,-\rangle_\mathbb{C}$ the standard complex bilinear pairing on $\mathbb{C}^n$.

\begin{definition}[Complex Constant Pivot]\label{Complex_constant_pivot}
For any $\phi\in\mathbb{R}[x_1,\dots,x_n]$, a complex constant pivot for $\phi$ is a nonzero constant vector {$\zeta\in\mathbb{C}^n$} for which {$\langle \operatorname{Hess}(\phi)\zeta,\zeta\rangle_\mathbb{C}$} is a constant.
\end{definition}

From now on, we write {$\zeta=(\zeta_1,\zeta_2,\zeta_3,\zeta_4)\in V_{\mathbb{C}}$}.

Complexification turns the pivot condition into the existence of a common zero of quadratic forms. The nonconstant part of {$D_\zeta^2\phi(x)$}, regarded as a quadratic form in {$\zeta$}, contains precisely the required information.

\begin{definition}[Hesse system]\label{Hesse_system}
For a polynomial function $\phi\in\mathbb{R}[x_1,\dots,x_4]$ and a point $c\in{V_{\mathbb{C}}}$, we regard {$D_\zeta^2\phi(c)$} as a quadratic form in {$\zeta$}. The linear system
$$
\mathcal{H}_\phi:=\operatorname{span}_{\mathbb{C}}\{{D_\zeta^2\phi(a)-D_\zeta^2\phi(b)}:a,b\in{V_{\mathbb{C}}}\}\subseteq\operatorname{Sym}^2({V_{\mathbb{C}}^*})
$$
is said to be the Hesse system of $\phi$, and it is equipped with the natural multiplication
$$
\mu_\phi\colon\operatorname{Sym}^3({V_{\mathbb{C}}^*})\otimes\mathcal{H}_\phi\rightarrow\operatorname{Sym}^5({V_{\mathbb{C}}^*})
$$
given by $\mu_\phi(C\otimes q):=Cq$.
\end{definition}

For any $\phi\in\mathbb{R}[x_1,\dots,x_4]$, there exist unique $q_0,\dots,q_N\in\operatorname{Sym}^2({V_{\mathbb{C}}^*})$ such that 
$$
{D_\zeta^2\phi(x)}=q_0(\zeta)+\sum_{i=1}^Nm_i(x)q_i(\zeta),
$$
where $m_1,\dots,m_N$ are the distinct nonconstant monomials which occur, and it is readily seen that the Hesse system is nothing but $\operatorname{span}_\mathbb{C}\{q_1,\dots,q_N\}$.

Thus, a nonzero vector {$\zeta$} is a complex constant pivot precisely when all the quadrics in $\mathcal{H}_\phi$ vanish at {$\zeta$}. The next proposition expresses this condition as the failure of a single multiplication mapping to be surjective.

\begin{proposition}\label{Complex_pivot_criterion}
A polynomial function $\phi\in\mathbb{R}[x_1,\dots,x_4]$ admits a complex constant pivot if and only if $\operatorname{rank}(\mu_\phi)\leq55$.
\end{proposition}

\begin{proof}
Put {$S_k:=\operatorname{Sym}^k(V_{\mathbb{C}}^*)$} and write
$$
{S:=\operatorname{Sym}(V_{\mathbb{C}}^*)=\mathbb{C}[\zeta_1,\zeta_2,\zeta_3,\zeta_4]}=\bigoplus_{k=0}^\infty S_k.
$$
A vector {$\zeta\in V_{\mathbb{C}}-\{0\}$} is a complex constant pivot if and only if {$[\zeta]\in\mathbb{P}(V_{\mathbb{C}})$} is a common zero of $\mathcal{H}_\phi$. If such a common zero {$[\zeta]$} exists, every element of {$\operatorname{im}(\mu_\phi)$} vanishes at {$\zeta$}, and hence the mapping $\mu_\phi$ is not surjective.

Conversely, we assume that $\mathcal{H}_\phi$ has no common zero in {$\mathbb{P}(V_{\mathbb{C}})$}. Then, four general elements $q_1,\dots,q_4\in\mathcal{H}_\phi$ form a regular sequence in $S$. The quotient ring $S/(q_1,\dots,q_4)$ has Hilbert series $$\frac{(1-t^2)^4}{(1-t)^4}=(1+t)^4,$$ and hence its homogeneous component of degree five vanishes. Therefore, we immediately obtain that
$$
S_5=S_3q_1+S_3q_2+S_3q_3+S_3q_4,
$$
whose right-hand side is contained in {$\operatorname{im}(\mu_\phi)$}, by the very definition. Thus, the mapping $\mu_\phi$ is surjective. Since $S_5$ has complex dimension $56$, the proof is completed.
\end{proof}

\subsection{The Complex-to-Real Principle}

For a real potential, the pivot obtained in this way may initially have nonreal coordinates. We now show that the Lorentzian condition nevertheless produces a real pivot.

We begin this section with the identities obtained by taking the real and imaginary parts of a complex pivot.

\begin{lemma}\label{complex_screen}
Let $\phi\in\mathbb{R}[x_1,\dots,x_4]$, and let {$\zeta=u+\sqrt{-1}v\in V_{\mathbb{C}}$} be a complex constant pivot for $\phi$. Then, either $\phi$ admits a real constant pivot, or the real vectors $u,v$ are linearly independent and it holds that
$$
D_uD_v{\phi_{\geq3}}=D_u^2{\phi_{\geq3}}-D_v^2{\phi_{\geq3}}=0.
$$
Moreover, if $f:=D_u^2{\phi_{\geq3}}$, then {we have that} $D_uf=D_vf=0$.
\end{lemma}

\begin{proof}
Since {$D_\zeta^2\phi$} is constant, comparison of the homogeneous pieces gives {$D_\zeta^2\phi_{\geq3}=0$}.

Suppose first that $u$ and $v$ are linearly dependent. Since {$\zeta\neq0$}, there exist a nonzero vector $\xi\in\mathbb{R}^4$ and a scalar $c\in\mathbb{C}^{\times}$ such that {$\zeta=c\xi$}. We then obtain that $D_\xi^2{\phi_{\geq3}}=0$.
Thus, the vector $\xi$ is a real constant pivot for $\phi$.

It remains to assume that $u$ and $v$ are linearly independent. Taking the real and imaginary parts of {$D_\zeta^2\phi_{\geq3}=0$} gives the first two identities. Differentiating the first identity in the directions $v$ and $u$, respectively, gives the final identity.
\end{proof}

If the two real parts do not already yield a pivot, they determine a two-dimensional screen. The following cylinder lemma will be applied to its scalar Hessian profile.

\begin{lemma}\label{flat_graph_cylinder}
Let $h\colon\mathbb{R}^2\rightarrow\mathbb{R}$ be a smooth function which is bounded from below. If
$$
\det(\operatorname{Hess}(h))=0,
$$
then there exists a nonzero constant vector $v\in\mathbb{R}^2$ such that $D_vh=0$. In particular, after a real linear change of coordinates, the function $h$ depends on at most one variable.
\end{lemma}

\begin{proof}
The graph of $h$ is complete, because its induced metric dominates the Euclidean metric on $\mathbb{R}^2$. Its Gaussian curvature is
$$
\frac{\det(\operatorname{Hess}(h))}{(1+|\nabla h|^2)^2}=0.
$$
By the Hartman--Nirenberg cylinder theorem \cite{HartmanNirenberg1959}, the graph is a cylinder. Hence, there exists a nonzero constant vector $(v,a)\in\mathbb{R}^2\times\mathbb{R}$ such that
$$
h(x+tv)=h(x)+ta
$$
for every $x\in\mathbb{R}^2$ and $t\in\mathbb{R}$. Since a graph contains no vertical line, we have that $v\neq0$. Since $h$ is bounded from below, we also obtain that $a=0$. The assertion follows.
\end{proof}

We now apply the preceding cylinder theorem to the screen determined by a complex pivot.

\begin{lemma}\label{screen_cylinder}
Let $\phi\in\mathbb{R}[x_1,\dots,x_4]$ be a unimodular Hessian potential of index $1$. Suppose that $\phi$ admits a complex constant pivot but no real constant pivot. Then, after a real linear change of coordinates, there exist variables $z:=(x,y)$ and {$s:=(s_1,s_2)$} such that
$$
{\phi(z,s)=\frac{1}{2}\langle B(x)s,s\rangle+\langle b(z),s\rangle+c(z)},
$$
where {the mappings} $b\colon\mathbb{R}^2\rightarrow\mathbb{R}^2$ and $c\colon\mathbb{R}^2\rightarrow\mathbb{R}$ are polynomial, while
$$
{B(x):=A+f(x)I_2}
$$
for a constant symmetric matrix $A$ and a nonconstant polynomial $f$ in one variable.
\end{lemma}

\begin{proof}
Let {$\zeta\in V_{\mathbb{C}}$} be a complex constant pivot for $\phi$. Then, by Lemma~\ref{complex_screen}, the vectors {$\Re(\zeta),\Im(\zeta)\in\mathbb{R}^4$} are linearly independent. 

Choose the coordinates {$s_1,s_2$} in directions {$\Re(\zeta),\Im(\zeta)$}, respectively, and complete them to a coordinate system {$x,y,s_1,s_2$}. The identities in Lemma~\ref{complex_screen}, together with Hessian integrability, then yield
$$
{\phi(z,s)=\frac{1}{2}\langle B(z)s,s\rangle+\langle b(z),s\rangle+c(z)},
$$
where {the matrix $B(z)$ is given by}
$$
{B(z):=A+f(z)I_2}
$$
for a constant symmetric matrix $A$. The polynomial $f$ is nonconstant, since otherwise {we would have that} either {$\Re(\zeta)$} or {$\Im(\zeta)$} would be a real constant pivot.

After an orthogonal change of the variables {$s=(s_1,s_2)$}, we may write
$$
{A=\operatorname{diag}(\lambda_1,\lambda_2).}
$$
The matrix {$B(z)$} is the restriction of {$\operatorname{Hess}(\phi)_{(z,s)}$} to the {screen plane}. Since a Lorentzian form admits no negative semi-definite two-dimensional subspace, we obtain that
$$
\max\{\lambda_1,\lambda_2\}>-f(z).
$$

Put $p:=\nabla f$ and {$H_f:=\operatorname{Hess}(f)$}. Also, put
$$
{\rho(f):=\langle\operatorname{adj}(H_f)p,p\rangle.}
$$
Over the rational function field $\mathbb{R}(x,y)$, the component of degree two in the variable {$s$} of the transverse Schur complement is nothing but
$$
{S_2(z,s):=\frac{|s|^2}{2}H_f(z)-\langle B(z)^{-1}s,s\rangle p\otimes p}.
$$
Since the Hessian determinant of $\phi$ is constant, the component of degree four in the variable {$s$} vanishes. Multiplying the identity {$\det(S_2)=0$} by {$\det(B(z))$} and comparing quadratic forms in {$s$}, we obtain that
$$
{\det(B(z))\det(H_f)I_2=2\rho(f)\operatorname{adj}(B(z))},
$$
which is a polynomial identity after clearing the denominator.

Suppose first that $\lambda_1\neq\lambda_2$. Comparison of the two diagonal entries gives {$\rho(f)=0$}. Since $f$ is nonconstant, the polynomial {$\det(B(z))$} does not vanish identically. Therefore, we obtain that
$$
\det(\operatorname{Hess}(f))=0.
$$
The polynomial $f$ is bounded from below, so Lemma~\ref{flat_graph_cylinder} shows that it depends on one real linear form.

Suppose now that $\lambda_1=\lambda_2$, and put $g:=\lambda_1+f$. The Lorentzian index shows that $g>0$ everywhere. The preceding matrix identity becomes
$$
g\det(\operatorname{Hess}(g))=2\langle\operatorname{adj}(\operatorname{Hess}(g))\nabla g,\nabla g\rangle.
$$
Consequently, we have that
$$
{\det(\operatorname{Hess}(1/g))=0.}
$$

The function $1/g$ is positive and smooth on the whole plane. Lemma~\ref{flat_graph_cylinder} gives a nonzero constant vector $\xi\in\mathbb{R}^2$ such that $D_\xi(1/g)=0$. Hence, we obtain that $D_\xi f=0$.

In both cases, a real linear change of the variables $z$ now gives $f=f(x)$, as asserted.
\end{proof}

The screen profile has therefore been reduced to one variable. The constant determinant identity next forces the potential to be quadratic in the remaining three variables.

\begin{lemma}\label{screen_quadratic_reduction}
Under the hypotheses and notations of Lemma~\ref{screen_cylinder}, after a further polynomial rearrangement, there exist variables {$u=(u_1,u_2,u_3)$} such that
$$
{\phi(x,u)=\frac{1}{2}\langle M(x)u,u\rangle+\langle \ell(x),u\rangle+r(x)},
$$
where {the matrix} $M$ is a 3 by 3 symmetric matrix-valued polynomial function and the entries of {$\ell$} and the function $r$ are polynomial in $x$.
\end{lemma}

\begin{proof}
We write $b=(b_1,b_2)$, and put 
$$
{J_b:=\begin{pmatrix}
\partial_xb_1 & \partial_yb_1 \\
\partial_xb_2 & \partial_yb_2
\end{pmatrix}}.
$$
Over the function field $\mathbb{R}(x,y)$, decompose the transverse Schur complement {$S$} of $\operatorname{Hess}(\phi)$ as
$$
{S=S_2+S_1+S_0,}
$$
where {the matrix $S_2$ is given by}
$$
{S_2=\frac{|s|^2}{2}\operatorname{Hess}(f)-\langle B(x)^{-1}s,s\rangle\nabla f\otimes\nabla f},
$$
and {the matrix $S_1$ is given by}
$$
{S_1=s_1\operatorname{Hess}(b_1)+s_2\operatorname{Hess}(b_2)-\nabla f\otimes(s^\top B(x)^{-1}J_b)-(J_b^\top B(x)^{-1}s)\otimes\nabla f}.
$$
Finally, it also holds that
$$
{S_0=\operatorname{Hess}(c)-J_b^\top B(x)^{-1}J_b}.
$$
Since $f=f(x)$, put
$$
{q(s):=f''\frac{|s|^2}{2}-|f'|^2\langle B(x)^{-1}s,s\rangle}.
$$
We then have that {$S_2=\operatorname{diag}(q,0)$}. Moreover, we claim that the quadratic form {$q$} is nonzero. Indeed, if $A$ is not a scalar matrix, the identity {$q=0$} would force both $f''=0$ and $f'=0$. If $A$ is a scalar matrix, the same identity above would give the ODE
$$
(\lambda_1+f)f''=2|f'|^2,
$$
which is impossible for a nonconstant polynomial by comparison of the highest degrees.

Comparison of the terms of degree three in the variable {$s$} in the constant identity
$$
{\det(B(x))\det(S)}=\det(\operatorname{Hess}(\phi))
$$
gives
$$
{q(s)(s_1\partial_y^2b_1+s_2\partial_y^2b_2)=0.}
$$
Consequently, there exist vector-valued polynomial functions $m:=m(x)$ and $n:=n(x)$ such that
$$
b(x,y)=m(x)y+n(x).
$$
Now, we write
$$
{\alpha(s,x):=\langle m'(x),s\rangle-f'(x)\langle B(x)^{-1}m(x),s\rangle}.
$$
Comparison of the terms of degree two in {$s$} gives
$$
{q(s)(\partial_y^2c-\langle B(x)^{-1}m,m\rangle)-\alpha(s,x)^2=0}.
$$
Differentiating this identity with respect to $y$, we obtain that {$\partial_y^3c=0$}.
Thus, the polynomial $\phi$ is quadratic in the three variables {$s_1,s_2,y$}, and the desired expression follows.
\end{proof}

It remains to analyse a polynomial family of quadratic forms in three variables depending on one parameter.

\begin{lemma}\label{one_parameter_quadratic_pivot}
Use the notations in Lemma~\ref{screen_quadratic_reduction}.
Assume that $\det(\operatorname{Hess}(\phi))$ is a nonzero constant. Then, the potential $\phi$ admits a real constant pivot.
\end{lemma}

\begin{proof}
Suppose first that $\det(M)=0$ as a polynomial. The rank of $M$ at a generic point is 2, since otherwise the full Hessian matrix of $\phi$ would be singular. Comparison of the quadratic terms in $u=(u_1,u_2,u_3)$ gives
$$
M'\operatorname{adj}(M)M'=0.
$$
Over the function field $\mathbb{R}(x)$, choose a nonzero vector $a(x)$ spanning the kernel of $M$. For some nonzero rational function $\lambda$, we have that
$$
\operatorname{adj}(M)=\lambda \,a\otimes a.
$$
The preceding identity then gives $M'a=0$. Differentiating $Ma=0$, we also obtain that $Ma'=0$. Since the kernel of $M$ is one-dimensional, there exists a rational function $\rho$ such that $a'=\rho a$. Thus, the projective class of $a$ is constant. We may therefore choose a nonzero constant vector $e\in\mathbb{R}^3$ such that
$$
{M(x)e=0.}
$$
The vector $e$ is a real constant pivot of $\phi$.

It remains to assume that $\det(M)$ is not the zero polynomial. Comparison of the quadratic terms in $u$ in the Schur-complement identity for $\operatorname{Hess}(\phi)$ gives
$$
M''=2M'M^{-1}M'.
$$
Differentiating twice, we obtain that $\partial_x^2M^{-1}=0$. Hence, there exist constant symmetric matrices $A,B$ such that
$$
M(x)^{-1}=A+xB.
$$
Since $M$ is a polynomial inverse of $M^{-1}$, the polynomial $\det(A+xB)$ is a nonzero constant. In particular, the matrix $B$ is singular. Choose a nonzero vector $a\in{\ker(B)}$ and put $e:=Aa$. Since $A=M(0)^{-1}$ is invertible, we have that $e\neq0$. It is also readily seen that $M(x)e=a$.

Consequently, we obtain that
$$
D_e^2\phi=\langle a,Aa\rangle,
$$
which is constant. Thus, in this case, the vector $e$ is again a real constant pivot for $\phi$.
\end{proof}

The preceding reductions now give the desired complex-to-real principle.

\begin{proposition}\label{Complex-to-real}
Let $\phi\in\mathbb{R}[x_1,\dots,x_4]$ be a unimodular Hessian potential of index $1$. Assume that $\phi$ admits a complex constant pivot. Then, in fact, the potential $\phi$ admits a real constant pivot.
\end{proposition}

\begin{proof}
Suppose, for a contradiction, that $\phi$ admits no real constant pivot. Lemma~\ref{screen_cylinder} gives the screen normal form with profile $f=f(x)$. Lemma~\ref{screen_quadratic_reduction} then shows that $\phi$ is quadratic in three variables with coefficients depending polynomially on one variable. Lemma~\ref{one_parameter_quadratic_pivot} now gives a real constant pivot, which is a contradiction.
\end{proof}

Consequently, the complex rank criterion is already a real Lorentzian pivot criterion.

\begin{corollary}[{Real Pivot Criterion}]\label{Complex_to_Real_pivot_criterion}
A unimodular Hessian potential $\phi\in\mathbb{R}[x_1,\dots,x_4]$ of index $1$ admits a constant pivot if and only if $\operatorname{rank}(\mu_\phi)\leq55$.
\end{corollary}

\begin{proof}
If $\phi$ admits a real constant pivot, then it admits a complex constant pivot, and Proposition~\ref{Complex_pivot_criterion} gives the rank inequality. Conversely, the rank inequality gives a complex constant pivot by Proposition~\ref{Complex_pivot_criterion}, and Proposition~\ref{Complex-to-real} then gives a real constant pivot.
\end{proof}

We next pass from the rank of $\mu_\phi$ to geometric conditions on the Hesse system itself.

\section{On Hesse Systems}

We use the dimension of the Hesse system to obtain a directly computable sufficient condition.

\subsection{\texorpdfstring{Low-Dimensional Hesse Systems}{Low-Dimensional Hesse Systems}}

If the Hesse system has dimension at most three, a common zero follows immediately from projective dimension. To include the borderline dimension four, we need the following two lemmas.

The first lemma states that four quadrics without a common zero force the Hessian-difference mapping to have Zariski-dense image.

\begin{lemma}\label{square_hesse_dominance}
Let $\phi\in\mathbb{R}[x_1,\dots,x_4]$ be a polynomial function such that
$$
{\dim_{\mathbb{C}}(\mathcal{H}_\phi)=4.}
$$
If the quadrics in $\mathcal{H}_\phi$ have no common zero in {$\mathbb{P}(V_{\mathbb{C}})$}, then the polynomial map {$\Delta_\phi\colon V_{\mathbb{C}}\rightarrow\mathcal{H}_\phi$} given by
$$
{\Delta_\phi(x):=H_\phi(x)-H_\phi(0)}
$$
is dominant.
\end{lemma}

The mapping {$\Delta_\phi$} simply records, at the point $x$, the nonconstant part of the Hessian as an element of the four-dimensional space $\mathcal{H}_\phi$. Here, dominant means that its image is Zariski dense in $\mathcal{H}_\phi$.

\begin{proof}
Choose a basis $q_1,q_2,q_3,q_4$ of $\mathcal{H}_\phi$. There exist polynomial functions $f_1,f_2,f_3,f_4$ on {$V_{\mathbb{C}}$} such that
$$
{\Delta_\phi(x)=\sum_{i=1}^4f_i(x)q_i}.
$$
Under this basis, the mapping {$\Delta_\phi$} is identified with $(f_1,f_2,f_3,f_4)$.

We first compare {$\Delta_\phi$} with the quadratic mapping determined by the four basis quadrics.

Regarding each $q_i$ as a symmetric bilinear form by polarization, define {$Q\colon V_{\mathbb{C}}\rightarrow\mathbb{C}^4$} via
$$
{Q(a):=(q_1(a,a),\dots,q_4(a,a)).}
$$
By the hypothesis, the identity {$Q(a)=0$} holds only when $a=0$. Fix a Hermitian norm on {$V_{\mathbb{C}}$}. Compactness of the unit sphere and homogeneity of {$Q$} give a constant $c>0$ such that {$|Q(a)|\geq c|a|^2$} for every {$a\in V_{\mathbb{C}}$}. It follows that every fibre of {$Q$} is compact. A positive-dimensional complex affine algebraic set cannot be compact. Therefore, the mapping {$Q$} has finite fibres.

The finiteness of {$Q$} controls the loci on which its differential drops rank. We now show that failure of dominance would produce a common zero of the Hesse system.

Suppose, for a contradiction, that {$\Delta_\phi$} is not dominant, so that its image is not Zariski dense in $\mathcal{H}_\phi$. Let $s\leq3$ be the generic rank of {$d\Delta_\phi$}. On the nonempty open subset on which {$d\Delta_\phi$} has rank $s$, put {$K:=\ker(d_x\Delta_\phi)$} and put $r:=4-s$.

Fix a vector {$a\in K$}. For any $b,c\in {V_{\mathbb{C}}}$, the symmetry of the third derivative shows that {$(d_x\Delta_\phi)b(a,c)$} equals {$(d_x\Delta_\phi)a(b,c)$}, while the latter expression vanishes because {$a\in K$}. On the other hand, we also have
$$
{(d_aQ)c=2(q_1(a,c),\dots,q_4(a,c))}.
$$
Using the standard complex bilinear pairing on $\mathbb{C}^4$, the pairing of {$(d_x\Delta_\phi)b$} with {$(d_aQ)c$} is twice {$(d_x\Delta_\phi)b(a,c)$}, and hence it vanishes. Thus, the image of {$d_aQ$} is orthogonal to the image of {$d_x\Delta_\phi$}. We therefore obtain {$\operatorname{rank}(d_aQ)\leq r$}.

Thus, every kernel {$K$} is contained in the rank-$r$ locus of {$dQ$}.

Let $Z_r$ be the common zero set of the minors of order $r+1$ of the matrix {$dQ$}. Thus, the locus $Z_r$ is precisely the algebraic subset of {$V_{\mathbb{C}}$} on which {$dQ$} has rank at most $r$. We have proved that {$K\subseteq Z_r$}. We claim that every irreducible component of $Z_r$ has dimension at most $r$. Indeed, since {$Q$} has finite fibres, its restriction to such a component has image of the same dimension. At a general smooth point, this dimension equals the rank of the differential of the restriction, which is at most $r$.

Since {$K$} is an $r$-dimensional linear subspace, it is therefore an irreducible component of $Z_r$. The algebraic subset $Z_r$ has only finitely many irreducible components. On every coordinate neighbourhood on which one fixed nonzero minor of order $s$ of {$d\Delta_\phi$} remains nonzero, the kernel {$K$} is obtained by solving the linear equations {$(d_x\Delta_\phi)a=0$}, and hence it depends algebraically on $x$. Since the constant-rank open subset is irreducible, while only finitely many components of $Z_r$ are available, we conclude that {$K$} is a fixed subspace {$K$}.

The kernels of {$d\Delta_\phi$} are therefore constant. A fixed direction in these kernels will annihilate every Hessian difference.

Now, choose a nonzero vector {$a\in K$}. We have that {$(d_x\Delta_\phi)a=0$} on a nonempty open subset, and hence this polynomial identity holds for every {$x\in V_{\mathbb{C}}$}. By the symmetry of the third derivative, every quadratic form belonging to the image of any {$d_x\Delta_\phi$} annihilates $a$. Finally, the images of all the differentials {$d_x\Delta_\phi$} span $\mathcal{H}_\phi$, because {$\Delta_\phi(0)=0$} and
$$
{\Delta_\phi(x)=\int_0^1(d_{tx}\Delta_\phi)x\,dt.}
$$
Consequently, every quadratic form in $\mathcal{H}_\phi$ vanishes at $a$, which contradicts the hypothesis.
\end{proof}

The second lemma is purely Lorentzian: a linear space of nilpotent self-adjoint endomorphisms has dimension at most three.

\begin{lemma}\label{lorentzian_nilpotent_bound}
Let $G$ be a non-singular 4 by 4 real symmetric matrix of index $1$. Let {$N$} be a linear subspace of the space of self-adjoint endomorphisms of $(\mathbb{R}^4,G)$, every member of which is nilpotent. Then, we have that
$$
{\dim_{\mathbb{R}}(N)\leq3.}
$$
\end{lemma}

\begin{proof}
For any {$A,B\in N$}, the endomorphisms $A$, $B$, and $A+B$ are nilpotent. We therefore obtain that
$$
{0=\operatorname{tr}((A+B)^2)=2\operatorname{tr}(AB).}
$$
Thus, the standard matrix trace pairing {$\langle A,B\rangle:=\operatorname{tr}(AB)$} vanishes on {$N$}.

Choose a basis of $\mathbb{R}^4$ in which $G=\operatorname{diag}(1,1,1,-1)$. Every self-adjoint endomorphism $T$ has the form $T=G^{-1}S$ for a unique symmetric matrix 
$$S=\begin{pmatrix}
    s_{11} & \dots  & s_{14} \\
    \vdots& \ddots & \vdots \\
    s_{41} & \dots  & s_{44}
\end{pmatrix}$$ and a direct computation then gives
$$
\operatorname{tr}(T^2)=s_{11}^2+s_{22}^2+s_{33}^2+s_{44}^2+2s_{12}^2+2s_{13}^2+2s_{23}^2-2s_{14}^2-2s_{24}^2-2s_{34}^2.
$$
Hence, the trace quadratic form $\operatorname{tr}(T^2)$ has signature $(7,3)$. Since the trace pairing vanishes on {$N$}, the space {$N$} is totally isotropic for this quadratic form. Using basic linear algebra, we conclude that {$\dim_{\mathbb{R}}(N)\leq3$}.
\end{proof}

We can now prove the dimension-four application. A 4-dimensional Hesse system would contradict the preceding Lorentzian bound, while the cases of lower-dimensional Hesse systems follow from dimension theory.

\begin{corollary}\label{low_dimensional_Hesse_system}
Let $\phi\in\mathbb{R}[x_1,\dots,x_4]$ be a unimodular Hessian potential of index $1$. Suppose that the Hesse linear system $\mathcal{H}_\phi$ of $\phi$ has complex dimension at most $4$. Then, the gradient mapping $\nabla\phi\colon\mathbb{R}^4\rightarrow\mathbb{R}^4$ is a polynomial automorphism.
\end{corollary}

\begin{proof}
For simplicity, we write $r:=\dim_{\mathbb{C}}(\mathcal{H}_\phi)$. Suppose first that $r\leq3$. By the projective dimension theorem, the quadrics in $\mathcal{H}_\phi$ have a common zero in {$\mathbb{P}(V_{\mathbb{C}})$}. Thus, the potential $\phi$ admits a complex constant pivot. Proposition~\ref{Complex-to-real} then gives a real constant pivot.

It remains to assume that $r=4$. Suppose, for a contradiction, that $\phi$ admits no real constant pivot. Proposition~\ref{Complex-to-real} shows that the quadrics in $\mathcal{H}_\phi$ have no common zero in {$\mathbb{P}(V_{\mathbb{C}})$}. Lemma~\ref{square_hesse_dominance} therefore shows that the mapping {$\Delta_\phi$} in that lemma is dominant.

Put {$G:=\operatorname{Hess}(\phi_2)$}, and let {$\mathcal{H}_\phi(\mathbb{R})$} be the real linear space spanned by the Hessian differences {$H_\phi(x)-G$} for $x\in\mathbb{R}^4$. Since polynomials with real coefficients are determined by their values on $\mathbb{R}^4$, under the standard identification of symmetric matrices with quadratic forms, we have that {$\mathcal{H}_\phi(\mathbb{R})\otimes_{\mathbb{R}}\mathbb{C}=\mathcal{H}_\phi$}. In particular, we have that {$\dim_{\mathbb{R}}(\mathcal{H}_\phi(\mathbb{R}))=4$}. Since the Hessian determinant of $\phi$ is constant and {$\Delta_\phi$} is dominant, the polynomial function sending $A$ to {$\det(G+A)-\det(G)$} vanishes on the entirety of $\mathcal{H}_\phi$.

In particular, for every {$A\in\mathcal{H}_\phi(\mathbb{R})$} and every $t\in\mathbb{R}$, we have
$$
{\det(I_4+tG^{-1}A)=1.}
$$
It follows that the characteristic polynomial of $G^{-1}A$ is $\lambda^4$, and hence the endomorphism $G^{-1}A$ is nilpotent. Since $A$ is symmetric, the endomorphism $G^{-1}A$ is also self-adjoint with respect to the bilinear form $G$. Lemma~\ref{lorentzian_nilpotent_bound} now forces {$\dim_{\mathbb{R}}(\mathcal{H}_\phi(\mathbb{R}))\leq3$}, which contradicts {$\dim_{\mathbb{R}}(\mathcal{H}_\phi(\mathbb{R}))=4$}.

Therefore, the potential $\phi$ admits a real constant pivot in every case. Proposition~\ref{constant_pivot_inversion} completes the proof.
\end{proof}

The preceding corollary completes the circle of ideas developed above. Constant pivots first reduce the inversion problem to three variables, while the Hesse system detects their existence by a finite-dimensional complex calculation. In particular, the Hessian conjecture holds whenever the Hessian differences of a Lorentzian unimodular potential span a complex vector space of dimension at most four. This condition is independent of the degree and can be checked directly from the coefficients of the potential.

\begin{proposition}\label{dual_square_theorem}
If {$\mathcal{H}_\phi$} contains no nonzero quadratic form of rank one over $\mathbb{C}$, then it holds that $\deg(\phi)\leq4$.
Consequently, if $\phi$ is Lorentzian unimodular and $\mathcal{H}_\phi$ contains no nonzero quadratic form of rank one, then {we have that} $\phi$ admits a constant pivot.
\end{proposition}

\begin{proof}
We use the natural pairing between {$\operatorname{Sym}^2(V_{\mathbb{C}}^*)$} and {$\operatorname{Sym}^2(V_{\mathbb{C}})$}, and denote $E:=\mathcal{H}_\phi^{\perp}$. 
Regard $E$ as a linear system of quadrics on {$\mathbb{P}(V_{\mathbb{C}}^*)$}. If $[\ell]$ were a common zero of $E$, then {we would have that}
$$
\langle\ell^2,\beta\rangle=0
$$
for every $\beta\in E$, and hence {we would obtain that} $\ell^2\in E^\perp=\mathcal{H}_\phi$, contrary to the hypothesis. Thus, the linear system $E$ is base-point-free.

Let {$S:=\operatorname{Sym}(V_{\mathbb{C}})$} be the polynomial ring of constant-coefficient differential operators. The ideal generated by $E$ has height four. Four general elements $\beta_1,\ldots,\beta_4\in E$ therefore form a homogeneous system of parameters. Since {$S$} is Cohen--Macaulay, the elements $\beta_1,\ldots,\beta_4$ form a regular sequence, and hence {we obtain that}
$$
{\operatorname{Hilb}_{S/(\beta_1,\ldots,\beta_4)}(t)}
=\frac{(1-t^2)^4}{(1-t)^4}
=(1+t)^4.
$$
Therefore, every homogeneous differential operator of order at least five belongs to the ideal $(\beta_1,\ldots,\beta_4)$.

Each $\beta_i$ annihilates $\phi$, and constant-coefficient differential operators commute. Hence every fifth derivative of $\phi$ vanishes, so {we have that} $\deg(\phi)\leq4$. The Lorentzian quartic theorem then gives a constant pivot.
\end{proof}

We now consider the different situation in which every member of the Hesse system is singular.

\subsection{High-Dimensional Hesse Systems}

\begin{lemma}\label{space_of_singular_symmetric_matrices}
Let {$L$} be a complex linear
subspace of {$\operatorname{Sym}^2(V_{\mathbb{C}}^*)$} such that every matrix in {$L$} is singular. Then, it holds that {$\dim_\mathbb{C}(L)\leq6$}. Moreover, if {$\dim_\mathbb{C}(L)=6$}, then a linear change of basis identifies {$L$} with $$\{\operatorname{diag}(A,0):A\in\operatorname{Sym}^2(\mathbb{C}^3)\}\subseteq{\operatorname{Sym}^2(V_{\mathbb{C}}^*)}.$$
\end{lemma}

\begin{proof}
Let $r$ be the largest rank of a matrix in {$L$}. The assertion is immediate if $r\leq1$. Suppose first that $r=2$. After a linear change of coordinates, we may assume that {$L$} contains the matrix {$A=\operatorname{diag}(I_2,0)$}. For any {$B\in L$}, write
$$
{B=\begin{pmatrix}C&U^\top\\ U&B_0\end{pmatrix}}.
$$
Since $\operatorname{rank}(A+tB)\leq2$ for every $t\in\mathbb{C}$, the Schur complement over $\mathbb{C}(t)$ gives $B_0=0$. If $B$ belongs to the kernel of the mapping that sends $B$ to $C$, then its rank is twice the rank of $U$. Therefore, in the above case, these matrices $U$ form a linear space of singular 2 by 2 complex matrices, and hence have dimension at most 2. Now, since $C$ is symmetric, we obtain that {$\dim_\mathbb{C}(L)\leq5$}.

It remains to assume that $r=3$. We may now take {$A=\operatorname{diag}(I_3,0)$} and write
$$
{B=\begin{pmatrix}C&u^\top\\ u&b_0\end{pmatrix}}.
$$
For every {$B\in L$}, we have that
$$
{0=\det(A+tB)=b_0\det(I_3+tC)t-\langle\operatorname{adj}(I_3+tC)u,u\rangle t^2}.
$$
It follows that $b_0=0$, and then the constant term of the remaining identity gives $\langle u,u\rangle=0$. Applying this conclusion to the sum of two members of {$L$}, we see that the image $E$ of the projection $\pi(B):=u$ is totally isotropic in $\mathbb{C}^3$. Thus, we have that $\dim_\mathbb{C}(E)\leq1$.

If $E=0$, every member of {$L$} annihilates the last coordinate vector, and hence we obtain that {$\dim_\mathbb{C}(L)\leq6$}. Equality gives precisely the stated linear space. 

Suppose finally that $E=\mathbb{C}u_0$, and let {$L_0:=\ker(\pi|_{L})$}. Fix a member of {$L$} whose last column is $u_0$. For every matrix {$\operatorname{diag}(C,0)\in L_0$}, the preceding determinant identity, applied after adding an arbitrary multiple of this matrix, gives
$$
\langle Cu_0,u_0\rangle=\langle\operatorname{adj}(C)u_0,u_0\rangle=0.
$$
After taking {$u_0=(1,\sqrt{-1},0)$}, the second expression above is
$$
{c_{33}(c_{22}-c_{11}-2\sqrt{-1}c_{12})+(c_{13}+\sqrt{-1}c_{23})^2}.
$$
This is a quadratic form of rank three on the 6-dimensional space of symmetric matrices $C$, whose restriction to {$L_0$} vanishes identically, and hence, its polarization vanishes on {$L_0$}. Its non-degenerate part has Witt index one, so we obtain that {$\dim_\mathbb{C}(L_0)\leq4$}. Therefore, we have that {$\dim_\mathbb{C}(L)\leq5$}. The proof is completed.
\end{proof}

\begin{corollary}\label{singular_Hesse_system}
Let $\phi\in\mathbb{R}[x_1,\dots,x_4]$ be a unimodular Hessian potential of index $1$. If every member of the Hesse system $\mathcal{H}_\phi$ is singular, then {we have that} $\dim_\mathbb{C}(\mathcal{H}_\phi)\leq6$. If $\dim_\mathbb{C}(\mathcal{H}_\phi)=6$, then {we have that} $\phi$ admits a constant pivot.
\end{corollary}

\begin{proof}
The first assertion follows immediately from Lemma~\ref{space_of_singular_symmetric_matrices}. 

Suppose now that $\dim_\mathbb{C}(\mathcal{H}_\phi)=6$. The equality statement gives a nonzero vector {$\zeta\in V_{\mathbb{C}}$} which belongs to the kernel of every member of $\mathcal{H}_\phi$. Hence, the vector {$\zeta$} is a complex constant pivot. Proposition~\ref{Complex-to-real} then gives a real constant pivot.
\end{proof}

In fact, using the algebraic geometry of linear systems of divisors, we can prove that even more is true:

\begin{proposition}\label{singular_Hesse_system_new}
Let $\phi\in\mathbb{R}[x_1,\dots,x_4]$ be a unimodular Hessian potential of index $1$. If every member of the Hesse system $\mathcal{H}_\phi$ is singular, then {we have that} $\phi$ admits a constant pivot.
\end{proposition}

\begin{proof}
If $\mathcal{H}_\phi=0$, the assertion is immediate. Otherwise, regard $\mathbb{P}(\mathcal{H}_\phi)$ as a projective linear system of quadric divisors on {$\mathbb{P}(V_{\mathbb{C}})$}. If the base locus of $\mathbb{P}(\mathcal{H}_\phi)$ were empty, Bertini's theorem would imply that a general member is smooth, contradicting the assumption that every member of $\mathcal{H}_\phi$ is singular. Therefore, the quadrics in $\mathcal{H}_\phi$ have a common zero {$[\zeta]$} in {$\mathbb{P}(V_{\mathbb{C}})$}. Any nonzero representative of this point {$[\zeta]$} is a complex constant pivot, and Proposition~\ref{Complex-to-real} gives a real constant pivot. The statement is proved.
\end{proof}

\begin{theorem}\label{dim_5/6_system_reduction}
Let $\phi\in\mathbb{R}[x_1,\dots,x_4]$ be a unimodular Hessian potential of index $1$. Suppose that the linear system $\mathbb{P}(\mathcal{H}_\phi)$ is {base-point-free}, and {suppose also that} $\dim_{\mathbb{C}}(\mathcal{H}_\phi)\leq6$. Then, there exist linearly independent vectors {$u,v\in V_{\mathbb{C}}$} such that $D_uD_v{\phi_{\geq3}}=0$. Consequently, there exist complex polynomial functions $P,Q$ in 3 variables, such that
$$
{\phi_{\geq3}}(z_1,\dots,z_4)=P(z_1,z_3,z_4)+Q(z_2,z_3,z_4)
$$
in some complex coordinate system $(z_1,\dots,z_4)$.
\end{theorem}

\begin{proof}
We use the natural pairing between {$\operatorname{Sym}^2(V_{\mathbb{C}}^*)$} and {$\operatorname{Sym}^2(V_{\mathbb{C}})$}, and denote $E:=\mathcal{H}_\phi^{\perp}$. 

Recall that the projective variety of symmetric tensors of rank at most two has dimension $6$ in {$\mathbb{P}(\operatorname{Sym}^2(V_{\mathbb{C}}))=\mathbb{P}^9$}. Since $\dim_{\mathbb{C}}(\mathcal{H}_\phi)\leq6$, we have that $\mathbb{P}(E)$ has dimension at least $3$. The intersection theory then yields a nonzero symmetric tensor $\beta$ of rank at most two in $E$.

This tensor $\beta$ cannot have rank 1. Indeed, otherwise it is proportional to {$a\otimes a$} for some nonzero vector {$a\in V_{\mathbb{C}}$}, and every quadric in $\mathcal{H}_\phi$ vanishes at {$[a]$}. This contradicts the {base-point-freeness} of $\mathbb{P}(\mathcal{H}_\phi)$. Thus, we obtain that $\beta$ has rank {$2$}. Over $\mathbb{C}$, it is proportional to $u\otimes v+v\otimes u$ for some linearly independent vectors {$u,v\in V_{\mathbb{C}}$}. Its orthogonality to $\mathcal{H}_\phi$ shows that $D_uD_v\phi$ is constant. Since the polynomial {$\phi_{\geq3}$} has no term of degree at most two, we obtain that $D_uD_v{\phi_{\geq3}}=0$.

After a $\mathbb{C}$-linear change of coordinates, we may always assume that {$D_u=\partial_1$ and $D_v=\partial_2$, where the operators $\partial_1$ and $\partial_2$ denote differentiation with respect to $z_1$ and $z_2$, respectively}. Clairaut's theorem on mixed derivatives shows that no monomial of {$\phi_{\geq3}$} contains both $z_1$ and $z_2$. Grouping the remaining monomials gives the asserted decomposition.
\end{proof}

The preceding applications complete the degree-free part of the argument. We finally return to homogeneous degree separation and strengthen the two-layer theorem.

\subsection{Multi-Layer Potentials Revisited}

In this particular section, we will use the following notations: 

\begin{definition}
For an $n\times n$ real matrix $M$, we define its $k$-th {weighted principal-minor polynomial} by
$$
{\sigma_k(M;y_1,\dots,y_n):=\sum_{\substack{I\subseteq\{1,\dots,n\}\\|I|=k}}\det(M_I)\prod_{i\in I}y_i\in\mathbb{R}[y_1,\dots,y_n]},
$$
where {the matrix $M_I$ is the principal submatrix of $M$ indexed by $I$}. {We also put}
$$
{\sigma_k(M):=\sigma_k(M;1,\dots,1)}.
$$
\end{definition}

For a symmetric matrix $G\in\operatorname{GL}(4,\mathbb{R})$ and pairwise non-parallel vectors $v_1,\dots,v_n\in \mathbb{R}^4$, we define {$V:=[v_1\ \cdots\ v_n]$ and $C_G:=V^\top G^{-1}V$}. Consider
\begin{equation}\label{perturbation}
\phi(x):=\frac{1}{2}\langle Gx,x\rangle+\sum_{i=1}^nf_i(\langle v_i,x\rangle)\in \mathbb{R}[x_1,x_2,x_3,x_4]
\end{equation}
for some univariate {polynomials satisfying}
$$
{f_1(t),\dots,f_n(t)\in t^3\mathbb{R}[t].}
$$
\begin{lemma}\label{ridge_cauchy_binet_formula}
Then, by the Cauchy--Binet formula, we have 
$$
{\frac{\det(\operatorname{Hess}(\phi))}{\det(G)}=1+\sum_{i=1}^4\sigma_i(C_G;f''_1(\langle v_1,x\rangle),\dots,f''_n(\langle v_n,x\rangle))}.
$$
Therefore, the polynomial function $\phi(x)$ is a solution to the Monge-Ampère equation
\begin{equation}\label{Monge–Ampère}
{\det(\operatorname{Hess}(\phi))\equiv\det(G)}
\end{equation}
if and only if
$$
{\sum_{i=1}^4\sigma_i(C_G;f''_1(\langle v_1,x\rangle),\dots,f''_n(\langle v_n,x\rangle))=0.}
$$
\end{lemma}

\begin{proof}
We put
$$
{\Lambda_f(x):=\operatorname{diag}(f_1''(\langle v_1,x\rangle),\dots,f_n''(\langle v_n,x\rangle))}.
$$
Then, we have that {$\operatorname{Hess}(\phi)=G+V\Lambda_f(x)V^\top$}. The Sylvester determinant identity gives
$$
{\frac{\det(G+V\Lambda_f(x)V^\top)}{\det(G)}=\det(I_n+\Lambda_f(x)C_G)}.
$$
Expansion by principal minors gives the displayed formula. Since {$\operatorname{rank}(C_G)\leq4$}, every principal minor of order larger than four vanishes.
\end{proof}

From now on, we will call the quadratic part of $\phi$ the \textit{base}, and call the functions of the form $f_i(\langle v_i,x\rangle)$ the \textit{ridges} of $\phi$. The polynomial $\phi$
itself will be termed a \textit{regular perturbation} of its base whenever equation~(\ref{Monge–Ampère}) is satisfied. {The polynomial $\phi_{\geq3}$ defined above is the sum of all ridges of $\phi$.}

\begin{remark}
We recall that, by elementary commutative algebra, up to an affine term, every real polynomial in 4 variables whose quadratic part has Hessian $G$ in fact can always be written as a function of the form (\ref{perturbation}). This is simply because every homogeneous polynomial of degree $j\geq3$ is a linear combination of $j$-th powers of real linear forms. Now, taking the union of these directions over all homogeneous pieces and combining proportional directions gives the desired representation.

As a consequence, the construction above generates all polynomial solutions to the Monge-Ampère equation~(\ref{Monge–Ampère}).
\end{remark}

\begin{lemma}\label{lorentz_hessian_rank_collapse}
Let $\phi\in\mathbb{R}[x_1,\dots,x_4]$ be a unimodular Hessian potential of index $1$. {Put $G:=\operatorname{Hess}(\phi_2)$ and $R_\phi:=\operatorname{Hess}(\phi_{\geq3})$.} Suppose that
\begin{equation}\label{determinant_equation}
{\det(G+\lambda R_\phi)\equiv\det(G)}
\end{equation}
for all $\lambda\in\mathbb{R}$. Then,
every member of the Hesse system $\mathcal{H}_\phi$ has rank at most $2$, and there exists a nonzero vector {$\xi\in\mathbb{R}^4$} such that $R_\phi(x)\xi=0$ for every $x\in\mathbb{R}^4$.
\end{lemma}

\begin{proof}
Lemma~\ref{lorentz_flat_pencil} gives {$\operatorname{rank}(R_\phi(x))\leq2$} for every $x\in\mathbb{R}^4$. By the classification theorem in \cite{deBondtSmallRank}, after a real linear change of coordinates and addition of an affine polynomial, either we have 
$$
{\phi_{\geq3}=F(y_1,y_2),}
$$
or we have
$$
{\phi_{\geq3}=g(t)+\sum_{i=1}^3 b_i(t)y_i}
$$
in a linear coordinate system $(t,y_1,y_2,y_3)$ of $\mathbb{R}^4$.
In the first case, every matrix {$R_\phi(x)$} annihilates the same two-dimensional subspace of $\mathbb{R}^4$, so both assertions follow. In the second case, the Hessian {$R_\phi(x)$} is of the form
$$
\begin{pmatrix}
g''(t)+b_1''(t)y_1+b_2''(t)y_2+b_3''(t)y_3&b'(t)\\
{b'(t)^\top}&0
\end{pmatrix}.
$$
Now, put $\theta:=dt$ and
$$
{\ell:=\sum_{i=1}^3b_i'(t)dy_i.}
$$
If all of $b_1'',b_2'',b_3''$ vanish, the absence of a quadratic homogeneous part reduces this to the first case; otherwise, investigating the coefficient of $\lambda$ in the determinant identity (\ref{determinant_equation}) gives
$$
{2\langle G^{-1}\theta,\ell\rangle+\left(g''+\sum_{i=1}^3 b_i''y_i\right)\langle G^{-1}\theta,\theta\rangle=0}.
$$
Comparison of the coefficients of $y_i$ gives $\langle G^{-1}\theta,\theta\rangle=0$, and then {we obtain that} {$\langle G^{-1}\theta,\ell\rangle=0$}. Thus, the nonzero vector {$\xi:=G^{-1}\theta$} satisfies {$R_\phi(x)\xi=0$} for every $x\in\mathbb{R}^4$.
\end{proof}

\begin{theorem}[High-Degree Multi-Layer Pivot]\label{high_degree_multilayer_pivot}
Let $\phi\in\mathbb{R}[x_1,\dots,x_4]$ be a unimodular Hessian potential of index $1$. Assume $$\phi=\phi_d+\cdots+\phi_{d-k}+\phi_2+\phi_1+\phi_0,$$ where $k\geq0$ and $d\geq4k+3$. Then, the potential $\phi$ admits a constant pivot.
\end{theorem}

\begin{proof}
{Put $G:=\operatorname{Hess}(\phi_2)$, $R_\phi:=\operatorname{Hess}(\phi_{\geq3})$ and $L_\phi:=G^{-1}R_\phi$.} Write
$$
{\frac{\det(G+\lambda R_\phi)}{\det(G)}=1+\sum_{s=1}^4\lambda^s\sigma_s(L_\phi)}.
$$
The homogeneous degrees occurring in {$\sigma_s(L_\phi)$} belong to the interval
$$
{J_s:=\{j\in\mathbb{Z}:(d-k-2)s\leq j\leq(d-2)s\}}.
$$
For every $1\leq s\leq3$, we have that
$$
{\min(J_{s+1})-\max(J_s)}=d-(s+1)k-2>0.
$$
Thus, the four intervals {$J_1,\dots,J_4$} are pairwise disjoint. Since {$\det(G+R_\phi)=\det(G)$}, comparison of the homogeneous degrees gives {$\sigma_1(L_\phi)=\sigma_2(L_\phi)=\sigma_3(L_\phi)=\sigma_4(L_\phi)=0$}. Lemma~\ref{lorentz_hessian_rank_collapse} then yields a nonzero vector {$\xi$} such that $R_\phi(x)\xi=0$ for every $x\in\mathbb{R}^4$. Consequently, the polynomial {$D_\xi^2\phi=\langle G\xi,\xi\rangle$} is constant, and hence {the vector} {$\xi$} is by definition a constant pivot.
\end{proof}

The interval separation above is only one way of distinguishing the determinant coefficients. The same argument works whenever the corresponding sets of degrees are disjoint.

We will consider the Minkowski sums of subsets of $\mathbb{Z}$, so in particular, for any $E\subseteq\mathbb{Z}$ and non-negative $m\in\mathbb{Z}$, we write $mE:=\{a_1+\cdots+a_m:a_1,\dots,a_m\in E\}${, with $0E:=\{0\}$}.

\begin{lemma}
Let $m\in\mathbb{Z}$ be a positive integer and let $\varnothing\neq E\subseteq\mathbb{Z}$. If $1\leq i\leq j\leq m$ and $iE\cap jE\neq\varnothing$, then $mE\cap(m+i-j)E\neq\varnothing$.
Consequently, 
$$
mE\cap\bigcup_{k=1}^{m-1}kE=\varnothing
$$
if and only if $E,2E,\dots,mE$ are pairwise disjoint. 
\end{lemma}

\begin{proof}
Choose $a\in iE\cap jE$ and $b\in(m-j)E$. Then, the element $a+b$ belongs to both $mE$ and $(m+i-j)E$. If two distinct members of $E,2E,\dots,mE$ intersect, write them as $iE$ and $jE$ with $i<j$. The first assertion gives an intersection of $mE$ with $(m+i-j)E$, where $1\leq m+i-j\leq m-1$. The converse is immediate.
\end{proof}

\begin{theorem}
Let $\phi\in\mathbb{R}[x_1,\dots,x_4]$ be a unimodular Hessian potential of index $1$. Assume the set {$I:=\{i\in\mathbb{Z}:\phi_{i+2}\neq0,\ i\geq1\}$} is {4-separated}, in the sense that the Minkowski sums {$I,2I,3I,4I$} are pairwise disjoint. Then, the potential $\phi$ admits a constant pivot.
\end{theorem}

\begin{proof}
{Put $G:=\operatorname{Hess}(\phi_2)$, $R_\phi:=\operatorname{Hess}(\phi_{\geq3})$ and $L_\phi:=G^{-1}R_\phi$.} Write
$$
{\frac{\det(G+\lambda R_\phi)}{\det(G)}=1+\sum_{s=1}^4\lambda^s\sigma_s(L_\phi)}.
$$
For every {$i\in I$}, the matrix {$\operatorname{Hess}(\phi_{i+2})$} is homogeneous of degree {$i$}. Hence, every homogeneous degree occurring in {$\sigma_s(L_\phi)$} belongs to {$sI$}. Since the sets {$I,2I,3I,4I$} are pairwise disjoint and {$\det(G+R_\phi)=\det(G)$}, comparison of homogeneous degrees gives {$\sigma_1(L_\phi)=\sigma_2(L_\phi)=\sigma_3(L_\phi)=\sigma_4(L_\phi)=0$}. Lemma~\ref{lorentz_hessian_rank_collapse} provides a nonzero vector {$\xi$} such that $R_\phi(x)\xi=0$ for every $x\in\mathbb{R}^4$. Consequently, the polynomial {$D_\xi^2\phi=\langle G\xi,\xi\rangle$} is constant, and hence {the vector} {$\xi$} is a constant pivot.
\end{proof}

\section{Hessian Metrics with Singular Residual}

We now prove the following result: for any unimodular Hessian potential $\phi\in\mathbb{R}[x_1,\dots,x_4]$ of index $1$, if ${R_\phi(x)}$ is singular for every $x\in\mathbb{R}^4$, then the potential $\phi$ admits a pivot.

The proof proceeds according to the generic rank of the residual Hessian. After treating the low-rank cases, we analyse the projective kernel image in generic rank three and exclude the remaining developable branch.

For any polynomial mapping ${f\colon \mathbb{A}^n\rightarrow \mathbb{A}^n}$, we denote by ${df}$ its {differential, identified with the Jacobian matrix}.

\subsection{The Low-Rank Cases}

In this section, we will use the following notation: 

Let $\phi\in\mathbb{R}[x_1,\dots,x_4]$ be a unimodular Hessian potential of index $1$ and $\deg(\phi)=d$. For our purpose, after omitting the affine part of $\phi$ and performing a linear change of coordinates, we may always assume $\phi_0=\phi_1=0$ and $2\phi_2=\langle Gx,x\rangle$ for $G:=\operatorname{diag}(-1,1,1,1)$. With the notation fixed above, we have {$R_\phi(x)=\operatorname{Hess}(\phi_{\geq3})_x$} {and} {$H_\phi(x)=G+R_\phi(x)$}. The rank of ${R_\phi(x)}$ at a generic point $x\in\mathbb{R}^4$ is said to be the generic rank of ${R_\phi}$.

We now refine the definition of a constant pivot.

\begin{definition}\label{pivot_new}
A nonzero constant vector ${\zeta\in V_{\mathbb{C}}}$ is a complex pivot of $\phi$ if ${D_\zeta^2\phi}$ is a constant function, and a nonzero constant vector {$\xi\in \mathbb{R}^4$} is 
\begin{enumerate}
    \item [1.] a pivot of $\phi$ if {$D_\xi^2\phi$} is a constant function;
    \item [2.] a strong pivot if ${R_\phi(x)\xi=0}$ for every $x\in \mathbb{R}^4$;
    \item [3.] null if {$\langle G\xi,\xi\rangle=0$}.
\end{enumerate}
\end{definition}

\begin{lemma}\label{residual_pivot_criterion}
A nonzero constant vector {$\xi\in \mathbb{R}^4$} is a pivot of $\phi$ if and only if
$$
{\langle R_\phi(x)\xi,\xi\rangle=0}
$$
for every $x\in \mathbb{R}^4$.
\end{lemma}

\begin{proof}
By the normalization of the quadratic part, we have that
$$
{D_\xi^2\phi=\langle G\xi,\xi\rangle+\langle R_\phi(x)\xi,\xi\rangle}.
$$
Therefore, the polynomial {$D_\xi^2\phi$} is constant if and only if ${\langle R_\phi(x)\xi,\xi\rangle}$ is constant.

Since ${R_\phi(0)=0}$, this constant must be zero.
\end{proof}

The preceding criterion allows us to treat the low-rank cases directly.

\begin{proposition}\label{low_rank}
If the generic rank of ${R_\phi}$ is at most two, then {the potential} $\phi$ admits a pivot.
\end{proposition}

\begin{proof}
The rank zero case is immediate. Suppose first that the generic rank is 1. Theorem~7.2 and Corollary~7.3 in \cite{deBondtSmallRank} give ${\phi_{\geq3}}=f({p(x)})+{L(x)}$ for a nonzero constant covector $p$ and polynomials $f,{L}$ with ${\deg(L)\leq1}$. It is readily seen that every nonzero vector in the kernel of $p$ is therefore a strong pivot.

Suppose now that the generic rank is 2. Again by Theorem~7.2 and Corollary~7.3 in \cite{deBondtSmallRank}, after an $\mathbb{R}$-linear change of coordinates and the omission of an affine summand, either {$\phi_{\geq3}$} depends on at most two linear forms, or
$$
{\phi_{\geq3}}(t,u_1,u_2,u_3)=g(t)+b_1(t)u_1+b_2(t)u_2+b_3(t)u_3
$$ for $g,b_1,b_2,b_3\in\mathbb{R}[t]$. In the first case, the matrix ${R_\phi}$ has a fixed two-dimensional kernel. In the second case, every nonzero vector {$\xi=(0,c)$} with $c\in\mathbb{R}^3$ satisfies ${\langle R_\phi(x)\xi,\xi\rangle=0}$. The assertion follows from Lemma~\ref{residual_pivot_criterion}.
\end{proof}

\subsection{The Kernel Field and the Apex}

It remains to consider the case in which ${R_\phi}$ has generic rank three. The gradient image of {$\phi_{\geq3}$} is then a hypersurface, and its normal direction at a generic point gives the kernel of ${R_\phi}$.

For a subset $X$ of the affine space $\mathbb{A}^n$, we will denote {its Zariski closure by} {$\overline{X}$}.

\begin{definition}
\label{canonical_kernel_data}
Suppose that the generic rank of ${R_\phi}$ is three. Put
$$
{F_\phi:=d\phi_{\geq3}\colon V_{\mathbb{C}}\longrightarrow V_{\mathbb{C}}^*,}
$$
and let ${Y_\phi\subseteq V_{\mathbb{C}}^*}$ be the Zariski closure of ${F_\phi(V_{\mathbb{C}})}$.
Since ${V_{\mathbb{C}}}$ is irreducible, so is ${Y_\phi}$. Since ${F_\phi}$ has generic rank three, the variety ${Y_\phi}$ has dimension three and is therefore an irreducible hypersurface. Since ${F_\phi}$ has real coefficients, the hypersurface ${Y_\phi}$ is preserved by complex conjugation. After multiplication by a nonzero constant, we may therefore choose an irreducible polynomial ${P_Y\in\operatorname{Sym}(V_{\mathbb{C}})}$ such that $${Y_\phi=\{y\in V_{\mathbb{C}}^*:P_Y(y)=0\}}.$$ The vector field ${\eta_\phi(x):=dP_Y(F_\phi(x))\in V_{\mathbb{C}}}$ is termed {the canonical polynomial kernel field} of {$\phi_{\geq3}$}.
We then define the projective kernel map
$$
{\gamma_\phi\colon V_{\mathbb{C}}\dashrightarrow\mathbb{P}(V_{\mathbb{C}})}
$$
of {$\phi_{\geq3}$}
by ${\gamma_\phi(x):=[\eta_\phi(x)]}$ whenever ${\eta_\phi(x)\neq0}$.

We say that the hypersurface ${Y_\phi}$ has an apex, if there is a nonzero constant {covector} ${a\in V_{\mathbb{C}}^*}$ such that
$$
{Y_\phi=Y_\phi+\mathbb{C}a},
$$
or equivalently, there is such a {covector} $a$ for which
$$
{a(\eta_\phi(x))=0.}
$$
\end{definition}

The equivalence is justified, because if ${a(\eta_\phi(x))=0}$ for every $x$, then {the derivative} ${D_aP_Y}$ vanishes on the dense subset ${F_\phi(V_{\mathbb{C}})}$ of ${Y_\phi}$, and hence {we obtain that} ${P_Y}$ divides ${D_aP_Y}$. Since ${\deg(D_aP_Y)<\deg(P_Y)}$, we obtain that ${D_aP_Y=0}$, which is equivalent to ${Y_\phi=Y_\phi+\mathbb{C}a}$. The converse is immediate.

\begin{lemma}\label{canonical_quasi_translation}
{The} kernel field ${\eta_\phi}$ satisfies ${R_\phi\eta_\phi=(d\eta_\phi)\eta_\phi=0}$, where {the symbol} ${d\eta_\phi}$ is the {differential} of ${\eta_\phi}$.
Consequently, it holds that
$$
{\begin{cases}
F_\phi(x+t\eta_\phi(x))=F_\phi(x),\\
\eta_\phi(x+t\eta_\phi(x))=\eta_\phi(x).
\end{cases}}
$$
\end{lemma}

\begin{proof}
Differentiating the identity ${P_Y(F_\phi(x))=0}$ gives ${R_\phi(x)\eta_\phi(x)=0}$. The chain rule also gives
$$
{d_x\eta_\phi=d^2P_Y(F_\phi(x))R_\phi(x)}.
$$
Consequently, we obtain that ${(d\eta_\phi)\eta_\phi=0}$.

Along an integral curve of ${\eta_\phi}$, the vector field ${\eta_\phi}$ is constant. Hence, the integral curve through the point ${x\in V_{\mathbb{C}}}$ is ${x+t\eta_\phi(x)}$. {This gives} ${\eta_\phi(x+t\eta_\phi(x))=\eta_\phi(x)}$. Along the same curve, the derivative of ${F_\phi}$ is ${R_\phi\eta_\phi=0}$. {This also gives} ${F_\phi(x+t\eta_\phi(x))=F_\phi(x)}$.
\end{proof}

The rank-three argument first treats the case in which the gradient image has an apex.

\begin{proposition}
\label{rank3_apex}
If ${Y_\phi:=\overline{F_\phi(V_{\mathbb{C}})}}$ has an apex, then {the potential} $\phi$ admits a pivot.
\end{proposition}

\begin{proof}
Let $m$ be the dimension of the space of apex directions of ${Y_\phi}$. The apex directions of ${Y_\phi}$ are precisely the projective image apices of ${F_\phi}$ in the terminology of \cite{deBondtSmallRank}. Hence, after a complex linear change of coordinates and the omission of an affine summand, Corollary~7.3 in \cite{deBondtSmallRank} gives the following three possibilities:

If $m=3$, then {the polynomial} {$\phi_{\geq3}$} depends on only three linear coordinates, and hence {the matrix} ${R_\phi}$ has a fixed kernel. If $m=2$, we may write
$$
{\phi_{\geq3}}=g(z_1,z_2)+b_1(z_1,z_2)u_1+b_2(z_1,z_2)u_2.
$$
for $g,b_1,b_2\in\mathbb{C}[z_1,z_2]$.
Every constant direction in the $u_1u_2$-plane is then a complex pivot.

It remains to consider $m=1$. In this case, we may write
$$
{\phi_{\geq3}(z,t)=g(t,\langle p(t),z\rangle)+\langle b(t),z\rangle},
$$
where {we put} ${z:=(z_1,z_2,z_3)}$ and {take} $p,b\in\mathbb{C}[t]^{\oplus3}$. Put ${s:=\langle p(t),z\rangle}$ and ${a:=\partial_s^2g(t,s)}$. Since the generic rank of ${R_\phi}$ is three, we have that $a\neq0$.

Let ${G_0}$ be the upper-left $3\times3$ block of $G$ with respect to the coordinates ${(z,t)}$. Direct differentiation, followed by changing ${z}$ to ${z+\lambda(p\times p')}$, shows that all blocks of ${H_\phi}$ except its lower-right entry are independent of $\lambda$. Comparing the coefficient of $\lambda$ in ${\det(H_\phi)=\det(G)}$ gives
$$
{\det(G_0+a\,p\otimes p)}
\left(
{(\partial_sg)}\langle p\times p',p''\rangle+\langle b'',p\times p'\rangle
\right)=0.
$$

Suppose first that ${\det(G_0+a\,p\otimes p)}$ is not identically zero. The second factor
$${(\partial_sg)}\langle p\times p',p''\rangle+\langle b'',p\times p'\rangle$$
then vanishes. Since $p\neq0$, the polynomials $t$ and ${s=\langle p(t),z\rangle}$ are algebraically independent. Differentiation with respect to $s$ therefore gives $a\langle p\times p',p''\rangle=0$. The identity $\langle p\times p',p''\rangle=0$ is the vanishing of the Wronskian of the three components of $p$. Since the ground field has characteristic zero, the Wronskian criterion provides a nonzero constant vector $e\in\mathbb{C}^3$ such that $\langle e,p\rangle=0$. It follows that $(e,0)$ is a complex pivot.

Suppose next that ${\det(G_0+a\,p\otimes p)=0}$. Evaluation at the origin gives ${\det(G_0)=0}$. Since $G$ is nondegenerate, the matrix ${G_0}$ has rank two. Thus, there are a nonzero vector $n\in\mathbb{C}^3$ and a nonzero constant ${\chi}$ such that
$$
{\operatorname{adj}(G_0)=\chi\, n\otimes n}.
$$
The rank-one determinant identity now gives
$$
{0=\det(G_0+a\,p\otimes p)=a\chi\langle n,p\rangle^2}.
$$
Consequently, we have that $\langle n,p\rangle=0$, and hence {the vector} $(n,0)$ is again a complex pivot. Proposition~\ref{Complex-to-real} completes the proof.
\end{proof}

\subsection{The Projective Kernel Map}

To analyse the remaining possibilities, we now need to study the image of the projective kernel mapping.

\begin{proposition}
\label{projective_reduction}
The projective kernel map ${\gamma_\phi}$ of {$\phi_{\geq3}$} satisfies that the complex dimension $k$ of ${\overline{\gamma_\phi(V_{\mathbb{C}})}}$ is at most 2,
and
\begin{enumerate}
    \item [1.] if $k=0$, then {we have that} ${Y_\phi:=\overline{F_\phi(V_{\mathbb{C}})}}$ has an apex;
    \item [2.] if $k=1$, then {we have that} ${\overline{\gamma_\phi(V_{\mathbb{C}})}}$ is a rational curve;
    \item [3.] if $k=2$ and ${Y_\phi}$ has no apex, then {we have that} ${\overline{\gamma_\phi(V_{\mathbb{C}})}}$ is a
developable surface.
\end{enumerate}
Moreover, in the case of $k=2$ above, the {projective dual} of the surface ${S_\phi:=\overline{\gamma_\phi(V_{\mathbb{C}})}}$ is a real rational curve contained in
$$
{\bigl\{[\ell]\in\mathbb{P}(V_{\mathbb{C}}^*):\langle G^{-1}\ell,\ell\rangle=0\bigr\}},
$$
and the Zariski closure of every irreducible component of a general fibre of the affine Gauss map ${\gamma_Y}$ of ${Y_\phi}$ is a null affine line whose direction runs over ${S_\phi^\vee}$.
\end{proposition}

\begin{proof}
Put ${B(x):=d^2P_Y(F_\phi(x))}$ {and} ${J:=d\eta_\phi=BR_\phi}$. By Lemma~\ref{canonical_quasi_translation}, the map sending $x$ to ${x+t\eta_\phi(x)}$ is a polynomial automorphism with inverse ${x-t\eta_\phi(x)}$. Consequently, we have that ${\det(I_4+tJ)=1}$.

Differentiating the identity ${F_\phi(x+t\eta_\phi(x))=F_\phi(x)}$ gives
$$
{R_\phi(x+t\eta_\phi(x))(I_4+tJ)=R_\phi(x)}.
$$
It follows that
$$
{H_\phi(x+t\eta_\phi(x))(I_4+tJ)=H_\phi(x)+tGJ}.
$$
Since ${\det(H_\phi)}$ is constant, we obtain that ${\det(H_\phi+tGJ)=\det(H_\phi)}$. Thus, the matrix ${H_\phi^{-1}GJ}$ is nilpotent.

Put ${\hat{R}_\phi:=R_\phi H_\phi^{-1}G}$. The identity ${\hat{R}_\phi=R_\phi-R_\phi H_\phi^{-1}R_\phi}$ shows that ${\hat{R}_\phi}$ is symmetric. Moreover, it has rank three and kernel ${\mathbb{C}\eta_\phi}$. Since ${B\hat{R}_\phi=JH_\phi^{-1}G}$ is similar to ${H_\phi^{-1}GJ}$, both ${BR_\phi}$ and ${B\hat{R}_\phi}$ are nilpotent.
Indeed, if ${\hat{R}_\phi v=0}$, then {we have that} ${H_\phi^{-1}Gv\in\mathbb{C}\eta_\phi}$. Since ${H_\phi\eta_\phi=G\eta_\phi}$, it follows that ${v\in\mathbb{C}\eta_\phi}$.

At a generic point, put ${U:=V_{\mathbb{C}}/\mathbb{C}\eta_\phi}$ and ${E:=\operatorname{Ann}(\eta_\phi)\subseteq V_{\mathbb{C}}^*}${, and denote the quotient map by} ${\pi_\eta\colon V_{\mathbb{C}}\rightarrow U}$. The natural pairing identifies ${E}$ with ${U^*}$. The matrices ${R_\phi}$ and ${\hat{R}_\phi}$ induce isomorphisms ${a,\hat a\colon U\rightarrow E}$, while ${B}$ induces a map ${b\colon E\rightarrow U}$. The maps ${ba}$ and ${b\hat a}$ are nilpotent. As ${a\colon U\rightarrow E}$ is invertible, we have that $\operatorname{rank}(b)\leq2$. Under the preceding identifications, {the differential factors through $U$ as}
$$
{(d_x\gamma_\phi)v=\pi_\eta(Jv)=ba(\pi_\eta(v)),}
$$
and hence {we obtain that}
$$
{k=\operatorname{rank}(d\gamma_\phi)=\operatorname{rank}(b)\leq2}.
$$

If $k=0$, the line ${\mathbb{C}\eta_\phi}$ is constant, and hence any nonzero constant {covector annihilating} ${\mathbb{C}\eta_\phi}$ is an apex of ${Y_\phi}$. If $k=1$, restriction to a suitable affine line gives a dominant rational map from ${\mathbb{P}_{\mathbb{C}}^1}$ to ${\overline{\gamma_\phi(V_{\mathbb{C}})}}$.

Suppose that $k=2$ and that ${Y_\phi}$ has no apex. {Recall that} ${S_\phi=\overline{\gamma_\phi(V_{\mathbb{C}})}}$. The surface ${S_\phi}$ is nondegenerate. Indeed, a hyperplane containing ${S_\phi}$ would give a nonzero constant {covector} $a\in V_{\mathbb{C}}^*$ such that $a(\eta_\phi(x))=0$ for every $x$, and hence an apex of ${Y_\phi}$. Let ${\ell}$ span the kernel of $b$. Since $ba$ and ${b\hat a}$ are nilpotent, after choosing dual bases of ${U}$ and ${E}$, we have that
$$
\det(a^{-1}+tb)=\det(a^{-1})
$$
and that
$$
{\det(\hat a^{-1}+tb)=\det(\hat a^{-1})}.
$$
The symmetric map $b$ has rank two, and hence ${\operatorname{adj}(b)=c\,\ell\otimes\ell}$ for some $c\neq0$. The coefficients of $t^2$ in the preceding identities are therefore ${c\,a^{-1}(\ell,\ell)}$ and ${c\,\hat a^{-1}(\ell,\ell)}$, respectively. Comparison of the coefficients of $t^2$ gives ${a^{-1}(\ell,\ell)=0}$ and ${\hat a^{-1}(\ell,\ell)=0}$.

For $\xi\in{E}$, choose $v$ such that ${R_\phi v=\xi}$. Since ${G(v+G^{-1}\xi)=H_\phi v}$, we obtain that
$$
{\hat a^{-1}\xi=a^{-1}\xi+\pi_\eta(G^{-1}\xi)}.
$$
Evaluation at ${\ell}$ gives $${\hat a^{-1}(\ell,\ell)=a^{-1}(\ell,\ell)+\langle G^{-1}\ell,\ell\rangle}.$$ Consequently, we have that ${\langle G^{-1}\ell,\ell\rangle=0}$.

The equality ${b\ell=0}$ gives ${B\ell\in\mathbb{C}\eta_\phi}$. For every $v\in{V_{\mathbb{C}}}$, symmetry gives $${\ell(Jv)=(R_\phi v)(B\ell)=0}.$$ Since ${\ell(\eta_\phi)=0}$, the covector ${\ell}$ annihilates ${\mathbb{C}\eta_\phi+\operatorname{im}(J)}$. At a point where {the differential} ${d\gamma_\phi}$ has rank two, the vector space ${\mathbb{C}\eta_\phi+\operatorname{im}(J)}$ has dimension three and is the tangent space to the affine cone over ${S_\phi}$ at ${\eta_\phi}$. Therefore, the covector ${\ell}$ defines the tangent plane to ${S_\phi}$ at ${[\eta_\phi]}$, and hence
$$
{S_\phi^\vee\subseteq
\bigl\{[\ell]\in\mathbb{P}(V_{\mathbb{C}}^*):\langle G^{-1}\ell,\ell\rangle=0\bigr\}}.
$$

Suppose that ${\dim_{\mathbb{C}}(S_\phi^\vee)=2}$. Since the dual null quadric is irreducible, the preceding inclusion is an equality. Projective biduality then gives that ${S_\phi}$ is the primal null quadric. We may choose ${P_Y}$ with real coefficients and work at a real generic point. The nilpotent map ${N:=ba}$ has rank two, so choose a real Jordan chain ${Ne_1=0}$, ${Ne_2=e_1}$ and ${Ne_3=e_2}$. Put ${e_0:=\eta_\phi}$ and choose lifts of $e_1,e_2,e_3$ to $\mathbb{R}^4$.
Since ${\operatorname{im}(N)}=\operatorname{span}_{\mathbb{R}}\{e_1,e_2\}$, the tangent space to the affine cone over ${S_\phi}$ is $\operatorname{span}_{\mathbb{R}}\{e_0,e_1,e_2\}$. The symmetry of ${aN}$ gives ${a(e_1,e_1)=a(e_1,e_2)=0}$ and ${a(e_1,e_3)=a(e_2,e_2)=:\alpha\neq0}$.
$$
{R_\phi=}
\begin{pmatrix}
0&0&0&0\\
0&0&0&\alpha\\
0&0&\alpha&\beta\\
0&\alpha&\beta&{\gamma}
\end{pmatrix},
$$
where {the constant} $\alpha\neq0$. Since ${S_\phi}$ is the primal null quadric, the preceding description of its tangent space gives $\langle Ge_0,e_i\rangle=0$ for $0\leq i\leq2$. Since $G$ is nondegenerate, we also have that $\langle Ge_0,e_3\rangle\neq0$. Since the tangent plane to the null quadric at ${[\eta_\phi]}$ is ${\{v\in V_{\mathbb{C}}:\langle Gv,\eta_\phi\rangle=0\}}$, the first row of $G$ has the form ${(0,0,0,\chi)}$, where {we have that} ${\chi\neq0}$. Expansion of the determinant identity gives
$$
{\det(H_\phi)-\det(G)=-\alpha\chi^2\langle Ge_1,e_1\rangle}.
$$
Thus, we obtain that $\langle Ge_1,e_1\rangle=0$. The plane spanned by ${\eta_\phi}$ and $e_1$ is therefore totally isotropic, which is impossible for a real form of signature $(3,1)$. Hence, we have that ${\dim_{\mathbb{C}}(S_\phi^\vee)\neq2}$.

If ${\dim_{\mathbb{C}}(S_\phi^\vee)=0}$, then {the surface} ${S_\phi}$ is a plane, contrary to its nondegeneracy. Indeed, the constant tangent plane then contains a dense open subset of ${S_\phi}$, so ${S_\phi}$ itself is that plane. Therefore, we conclude that ${\dim_{\mathbb{C}}(S_\phi^\vee)=1}$, which implies that ${S_\phi}$ is developable. The Gauss map of ${S_\phi}$, composed with ${\gamma_\phi}$, dominates ${S_\phi^\vee}$. Moreover, this Gauss map of ${S_\phi}$ is a rational map defined over $\mathbb{R}$ and has rank one at a real generic point. Its restriction to a suitable real affine line is therefore nonconstant and gives a dominant rational map ${\mathbb{P}_{\mathbb{R}}^1\dashrightarrow S_\phi^\vee}$. L\"uroth's theorem shows that ${S_\phi^\vee}$ is a real rational curve.

Finally, consider the affine Gauss map $\gamma_Y\colon Y_\phi\dashrightarrow S_\phi$ given by $\gamma_Y(y):=[d_yP_Y]$.
At ${y=F_\phi(x)}$, the tangent space of ${Y_\phi}$ is ${E}${.}

{The tangent space of $S_\phi$ at $[\eta_\phi]$ is naturally identified with a subspace of $U$. Under this identification,} the differential ${d_y\gamma_Y}$ is the map $b$. Let ${\Sigma}$ be an irreducible component of a general fibre of ${\gamma_Y}$. The tangent line of ${\Sigma}$ is ${\mathbb{C}\ell}$. Since ${[\ell]}$ is the tangent plane to ${S_\phi}$ at the fixed point ${[\eta_\phi]=\gamma_Y(\Sigma)}$, this direction is constant along ${\Sigma}$. Every linear function which annihilates ${\ell}$ is therefore constant on ${\Sigma}$. Thus, we obtain that ${\Sigma}$ is contained in an affine line with direction ${\ell}$, and equality follows after taking the Zariski closure. Since ${[\ell]\in S_\phi^\vee}$ and ${\langle G^{-1}\ell,\ell\rangle=0}$, the affine line ${\Sigma}$ is null.
\end{proof}

The cases in which this image has dimension at most one reduce to the apex case by the following fact about quasi-translations.

\begin{lemma}
\label{projective_curve_case}
Let ${\eta:V_{\mathbb{C}}\rightarrow V_{\mathbb{C}}}$ be a polynomial map satisfying ${(d\eta)\eta=0}$.
Suppose that  
$$
{\dim_\mathbb{C}\overline{\{[\eta(x)]:\eta(x)\neq0\}}\leq1}.
$$
Then, it holds that
$$
{\dim_{\mathbb{C}}(\operatorname{span}_{\mathbb{C}}(\eta(V_{\mathbb{C}})))\leq3.}
$$
\end{lemma}

\begin{proof}
The hypothesis ${(d\eta)\eta=0}$ means that ${x+\eta}$ is a quasi-translation in the sense of \cite{deBondtQT}. Put ${Z_\eta:=\overline{\eta(V_{\mathbb{C}})}}$. Since the fibres of the projectivisation of ${Z_\eta-\{0\}}$ have dimension at most 1, we have that ${\operatorname{rank}(d\eta)=\dim_{\mathbb{C}}(Z_\eta)\leq2}$. If the generic rank of ${d\eta}$ is 0, the assertion is immediate. If the generic rank of ${d\eta}$ is 1, then the assertion follows from Theorem~5.5 in \cite{deBondtQT}.

Suppose that the generic rank is 2. The projective image then has dimension 1, so the differential of projectivisation gives ${\eta\in\operatorname{im}(d\eta)}$ at a generic point. Let ${I(Z_\eta)}$ be the prime ideal of ${Z_\eta}$, and let ${E}$ be the Euler derivation in the target variables. For every $f\in{I(Z_\eta)}$, choose a rational vector field $\xi$ such that ${(d\eta)\xi=\eta}$. We then obtain that
$$
{(Ef)(\eta)=D_\xi(f(\eta))=0}.
$$
Thus, the ideal ${I(Z_\eta)}$ is stable under ${E}$. Polynomial interpolation in the eigenvalues of ${E}$ shows that each homogeneous component of every element of ${I(Z_\eta)}$ also belongs to ${I(Z_\eta)}$. Hence, the variety ${Z_\eta}$ is an affine cone.

Put ${d:=\max\{\deg(\eta_1),\dots,\deg(\eta_4)\}}$ and homogenise ${\eta}$ by
$$
{\eta^{\mathrm h}(x,z):=\bigl(z^d\eta(x/z),0\bigr)}.
$$
Proposition~5.1 in \cite{deBondtQT} shows that ${\eta^{\mathrm h}}$ is a homogeneous quasi-translation in dimension five. Since ${Z_\eta}$ is a cone, we have that
$$
{\overline{\eta^{\mathrm h}(\mathbb{C}^5)}=Z_\eta\times\{0\}}.
$$
Consequently, the generic rank of ${d\eta^{\mathrm h}}$ is two. Theorem~5.3 in \cite{deBondtQT} gives
$$
{\dim_{\mathbb{C}}(\operatorname{span}_{\mathbb{C}}(\eta^{\mathrm h}(\mathbb{C}^5)))\leq3}.
$$
Restriction to $z=1$ proves the assertion.
\end{proof}

\begin{corollary}
\label{curve_branch}
If the image of the projective kernel map ${\gamma_\phi}$ has complex dimension at most one, then {the hypersurface} ${Y_\phi}$ has an
apex.
\end{corollary}

\begin{proof}
Lemma~\ref{projective_curve_case} gives a nonzero constant {covector $a\in V_{\mathbb{C}}^*$} such that ${a(\eta_\phi(x))=0}$ for every ${x\in V_{\mathbb{C}}}$. By Definition~\ref{canonical_kernel_data}, the hypersurface ${Y_\phi}$ has an apex.
\end{proof}

\subsection{The Developable Surface Branch}

Now, only the apex-free developable branch remains in the generic rank-three argument.

\begin{lemma}
\label{developable_rational_coordinates}
Suppose that the apex-free developable branch in Case 3 of Proposition~\ref{projective_reduction} occurs. For ${S_\phi:=\overline{\gamma_\phi(V_{\mathbb{C}})}}$, there exist a real rational parameter $s${,} {a real rational covector} ${\ell}$ and real rational vectors $n,m,e$ depending only on $s$ satisfying the following properties: For ${n:=G^{-1}\ell}$ and $m:=\partial_sn$, it holds that ${\mu:=\langle Gm,m\rangle\neq0}$ and
$$
{\ker(\ell)\cap\ker(\partial_s\ell)}
=\operatorname{span}_{\mathbb{R}(s)}\{n,e\}.
$$
Moreover, one may require that $\langle G\partial_se,m\rangle=0$. After replacing the kernel field by a nonzero invariant real rational multiple, there is a real rational function $r$ such that ${\hat{\eta}_\phi:=n+re}$ and ${\mathbb{R}(S_\phi)=\mathbb{R}(s,r)}$. It also holds that ${(dF_\phi)\hat{\eta}_\phi=(d\hat{\eta}_\phi)\hat{\eta}_\phi=0}$ and ${\det(I_4+t\,d\hat{\eta}_\phi)=1}$.

Furthermore, there exist a finite formally real extension ${L_0/\mathbb{R}(s,r)}$ and an element $\lambda$ transcendental over ${L_0}$, such that a general point $y$ of ${Y_\phi}$ is ${y=c+\lambda\ell(s)}$ for some {covector $c\in L_0\otimes_{\mathbb{R}}V^*$}. In particular, it holds that ${\mathbb{R}(Y_\phi)=L_0(\lambda)}$.
\end{lemma}

\begin{proof}
Proposition~\ref{projective_reduction} gives that the dual curve ${S_\phi^\vee}$ is real rational. Write ${\mathbb{R}(S_\phi^\vee)=\mathbb{R}(s)}$ and choose ${\ell=\ell(s)\in\mathbb{R}(s)\otimes_{\mathbb{R}}V^*}$. Since ${\ell}$ is null, we have that $\langle Gn,n\rangle=\langle Gm,n\rangle=0$. If ${\mu=0}$, then {the vectors} $n$ and $m$ span a totally isotropic plane. Since the real Lorentzian form has Witt index one, the vector $m$ is proportional to $n$, contrary to the non-constancy of ${[\ell]}$. Thus, we obtain that ${\mu\neq0}$.

Projective biduality identifies the general ruling of ${S_\phi}$ with ${\mathbb{P}(\ker(\ell)\cap\ker(\partial_s\ell))}$. Choose $e=e(s)$ so that this kernel is spanned by $n,e$. Replacing $e$ by ${e-\langle G\partial_se,m\rangle\mu^{-1}n}$, and then changing the ruling parameter, gives $\langle G\partial_se,m\rangle=0$. Since the ruling contains $[n(s)]$, we obtain that ${\mathbb{R}(S_\phi)=\mathbb{R}(s,r)}$ and its general point is $[n+re]$.

Let ${\eta_\phi}$ be the {canonical polynomial} kernel field in Definition~\ref{canonical_kernel_data}. Write ${\eta_\phi=\rho(n+re)}$ for a nonzero real rational function $\rho$. The functions $s,r$ are constant along the integral curves of ${\eta_\phi}$. The equality ${(d\eta_\phi)\eta_\phi=0}$ therefore gives ${(d\rho)\eta_\phi=0}$.

{Put} ${\hat{\eta}_\phi:=\rho^{-1}\eta_\phi}${. Then}
$$
{(dF_\phi)\hat{\eta}_\phi=(d\hat{\eta}_\phi)\hat{\eta}_\phi=0}.
$$
The matrix determinant lemma, together with
$$
{\begin{cases}
(d\eta_\phi)\eta_\phi=0,\\
(d(\rho^{-1}))\eta_\phi=0,\\
\det(I_4+t\,d\eta_\phi)=1,
\end{cases}}
$$
gives
$$
{\det(I_4+t\,d\hat{\eta}_\phi)=1}.
$$

Let ${L_0}$ be the relative algebraic closure of $\mathbb{R}(s,r)$ in ${\mathbb{R}(Y_\phi)}$. The extension ${L_0/\mathbb{R}(s,r)}$ is finite and separable. It is formally real, because the pullback by ${F_\phi}$ embeds ${\mathbb{R}(Y_\phi)}$ into $\mathbb{R}(x_1,x_2,x_3,x_4)$. The extension ${\mathbb{R}(Y_\phi)/L_0}$ is regular. Proposition~\ref{projective_reduction} shows that the closure of its geometric generic fibre is an affine line with direction ${\ell(s)}$. This line is defined over ${L_0}$, and its point at infinity is defined over $\mathbb{R}(s,r)$. Intersecting it with a fixed real affine hyperplane transverse to ${\ell}$ gives a point ${c\in L_0\otimes_{\mathbb{R}}V^*}$. Thus, we obtain that ${y=c+\lambda\ell(s)}$ and ${\mathbb{R}(Y_\phi)=L_0(\lambda)}$ for an element $\lambda$ transcendental over ${L_0}$.
\end{proof}

\begin{lemma}
\label{finite_differential_constants}
Let $K$ be a field of characteristic zero, let ${t}$ be transcendental over $K$, and let $N$ be a finite extension of ${K(t)}$. Extend the derivation ${\partial_t}$ to $N$, and put ${N_0:=\{a\in N:\partial_t a=0\}}$. 
Then, the field $N_0$ is finite over $K$, and {the element} ${t}$ is transcendental over $N_0$. Moreover, every derivation of ${K(t)}$ which commutes with ${\partial_t}$ extends uniquely to $N$ and preserves $N_0$.
\end{lemma}

\begin{proof}
Let $a\in N_0$, and differentiate its monic minimal polynomial over ${K(t)}$. Minimality shows that every coefficient belongs to $K$. Thus, the extension $N_0/K$ is algebraic. The element ${t}$ is transcendental over $N_0$, since differentiation of a minimal relation would give a nonzero relation of smaller degree. For every finite extension $L\subseteq N_0$ of $K$, the degree $[L:K]$ is bounded by ${[N:K(t)]}$. Taking composita shows that $N_0/K$ is finite. Finally, derivations extend uniquely through finite separable extensions. Their commutators vanish on ${K(t)}$, and hence on $N$. This proves the remaining assertion.
\end{proof}

\begin{proposition}
\label{developable_differential_reduction}
Suppose that the apex-free developable branch in Case 3 of
Proposition~\ref{projective_reduction} occurs.  Then there exist 
\begin{enumerate}
    \item [1.] a formally real field $K$ of characteristic zero;
    \item [2.] a nonzero element ${c_0\in K}$;
    \item [3.] an element $r$ transcendental over $K$;
    \item [4.] {a curve} ${C}$ {with function field} ${E:=K(r,w)}$ subject to the relation ${r^3+3rw^2=4c_0}$.
\end{enumerate}
Moreover, there also exist a finite separable differential extension ${L/E}$ and an
element ${g\in L}$ such that
$\partial_r g=w$.
Here, the derivation $\partial_r$ is trivial on $K$ and satisfies $\partial_r r=1$.
\end{proposition}

\begin{proof}

Use the notations in Lemma~\ref{developable_rational_coordinates}. {Write} ${\delta_\phi f:=(df)\hat{\eta}_\phi}$ for ${f\in\mathbb{R}(x_1,\dots,x_4)}$ {and let} ${M:=\ker\bigl(\delta_\phi\colon\mathbb{R}(x_1,\dots,x_4)\rightarrow\mathbb{R}(x_1,\dots,x_4)\bigr)}$.
The equality ${(dF_\phi)\hat{\eta}_\phi=0}$ gives ${\mathbb{R}(Y_\phi)\subseteq M}$.

Choose an index {$j$} for which the component ${(\hat{\eta}_\phi)_j\neq0}$. Put $X:=x-t\hat{\eta}_\phi$ and 
$$
t:=x_j/(\hat{\eta}_\phi)_j.
$$
We have that
$$
{\begin{cases}
\delta_\phi t=1,\\
\delta_\phi X=0,\\
\mathbb{R}(x_1,\dots,x_4)=M(t).
\end{cases}}
$$
The element $t$ is transcendental over ${M}$.

Since $s,r,\lambda,t$ are algebraically independent, the field ${\mathbb{R}(x_1,\dots,x_4)}$ is finite over ${\mathbb{R}(Y_\phi)(t)}$. Therefore, every ${\mathbb{R}(Y_\phi)}$-linearly independent subset of ${M}$ remains independent over ${\mathbb{R}(Y_\phi)(t)}$. Thus, the extension ${M/\mathbb{R}(Y_\phi)}$ is finite and separable.

Extend $\partial_s$ and $\partial_r$ from $\mathbb{R}(s,r)$ to ${L_0}$, and then to ${\mathbb{R}(Y_\phi)}$ by requiring that they annihilate $\lambda$. Let $\partial_\lambda$ annihilate ${L_0}$ and satisfy $\partial_\lambda\lambda=1$. The unique extensions of $\partial_s,\partial_r,\partial_\lambda$ to ${M}$ commute. Put ${L:=\ker(\partial_\lambda\colon M\rightarrow M)}$. Then, Lemma~\ref{finite_differential_constants} shows that ${L/L_0}$ is finite and that $\partial_s,\partial_r$ preserve ${L}$. Thus, ${L/\mathbb{R}(s,r)}$ is finite. Extend these derivations to ${\mathbb{R}(x_1,\dots,x_4)=M(t)}$ by requiring that they annihilate $t$. Then, {we have that} ${\delta_\phi=\partial_t}$, and the four commuting derivations form a basis at a generic point.

Put ${(\partial_1,\partial_2,\partial_3):=(\partial_s,\partial_r,\partial_\lambda)}$. For $i,j\in\{1,2,3\}$, put
$$
\begin{cases}
{R_{ij}:=(\partial_i y)\partial_jX},\\
{T_{ij}:=(\partial_i y)\partial_j\hat{\eta}_\phi},\\
{G_{ij}:=\langle G^{-1}\partial_i y,\partial_jy\rangle}.
\end{cases}
$$
Since ${\partial_i y=R_\phi(x)\partial_i x}$, the matrix ${R+tT}$ is symmetric. Denote ${J:=d_X\hat{\eta}_\phi}$. Since ${\hat{\eta}_\phi}$ is independent of $t$, we have that
$$
{(\partial_1x,\partial_2x,\partial_3x,\hat{\eta}_\phi)
=(I_4+tJ)(\partial_1X,\partial_2X,\partial_3X,\hat{\eta}_\phi)}.
$$
The covectors ${\partial_1y,\partial_2y,\partial_3y}$ form a basis of {$\operatorname{Ann}(\hat{\eta}_\phi)$}.
{Pair the preceding frames with these covectors and a covector dual to} ${\hat{\eta}_\phi}${.} {Using}
$$
{\det(I_4+tJ)=1},
$$
{we obtain the first determinant identity}
$$
{\det(R+tT)=\det(R)}{,}
$$
{and see that} ${R}$ is non-singular.

Now, denote ${\Theta_\phi:=x+G^{-1}F_\phi(x)}$. Since ${d\Theta_\phi=I_4+G^{-1}R_\phi}$, we have that ${\det(d\Theta_\phi)=1}$ and ${(d\Theta_\phi)\hat{\eta}_\phi=\hat{\eta}_\phi}$. Applying ${d\Theta_\phi}$ to the frame ${(\partial_1x,\partial_2x,\partial_3x,\hat{\eta}_\phi)}$ gives the second determinant identity ${\det(R+tT+G)=\det(R+tT)}$.

Finally, {the matrix $T$ represents $d\gamma_Y$ in these frames and has rank two by Proposition~\ref{projective_reduction}.}

Consequently, the matrices ${R}$ and ${T}$ are symmetric and
$$
\begin{cases}
{\det(R+tT)=\det(R)},\\
{\det(R+tT+G)=\det(R+tT)}.
\end{cases}
$$

The expressions for ${\ell,y,\hat{\eta}_\phi}$ in Lemma~\ref{developable_rational_coordinates} show that ${T}$ has vanishing third row and third column, whilst its {upper-left} block has the form
$$
{C=}
\begin{pmatrix}
{\alpha+\lambda\mu}&\beta\\
\beta&{\chi}
\end{pmatrix},
$$
where {we put} ${\beta:=(\partial_sc)e}$ and ${\chi:=(\partial_rc)e}$, and {we have that} ${\alpha,\beta,\chi\in L_0}$. The matrix ${C}$ is non-singular because ${T}$ has rank two. We denote ${v_\lambda:=(R_{13},R_{23})}$. 
Comparison of the coefficients of $t^2$ and $t$ gives
$$
{R_{33}=\langle\operatorname{adj}(C)v_\lambda,v_\lambda\rangle=0.}
$$
The vector ${v_\lambda}$ is nonzero because ${R}$ is non-singular.

Now, we put
$$
{h:=y(X)-\phi_{\geq3}(X)}.
$$
Since ${F_\phi(X)=y}$, the element $h$ belongs to ${M}$ and satisfies
$$
{dh=\langle dy,X\rangle}.
$$
Thus, it holds that ${\partial_\lambda^2h=R_{33}=0}$.

Put ${h_1:=\partial_\lambda h}$ and ${h_0:=h-\lambda h_1}$. Then, we have ${h_0,h_1\in L}$ and ${h=h_0+\lambda h_1}$. The element $\lambda$ is transcendental over ${L}$.

The covectors ${\ell,\partial_1y,\partial_2y}$ form a basis of ${\operatorname{Ann}(\hat{\eta}_\phi)}$. Their pairings with $X$ belong to ${L(\lambda)}$ and determine the class of $X$ modulo ${\hat{\eta}_\phi}$ over this field. Since ${\hat{\eta}_\phi}$ is independent of $\lambda$, differentiation gives ${v_\lambda\in L(\lambda)^{\oplus2}}$.

The nonzero vector ${v_\lambda}$ is isotropic for the non-singular binary form ${\operatorname{adj}(C)}$. The discriminant of this form is ${-4\det(C)}$, and hence {we obtain that} ${-\det(C)}$ is a square in ${L(\lambda)}$. We have that
$$
{\det(C)=(\lambda\mu+\alpha)\chi-\beta^2}.
$$
If ${\chi\neq0}$, then {the expression} ${-\det(C)}$ is a nonconstant affine polynomial in $\lambda$. Its valuation at infinity is odd, which is impossible for a square. Consequently, we obtain that ${\chi=0}$. Since ${C}$ is non-singular, we also obtain that $\beta\neq0$.

The normality identities ${(\partial_i y)\hat{\eta}_\phi=0}$ give ${G_{13}=-r\beta}$ and ${G_{23}=G_{33}=0}$.
The coefficient of $t$ in the second determinant identity, after subtracting ${\langle\operatorname{adj}(C)v_\lambda,v_\lambda\rangle=0}$, gives ${2r\beta^2R_{23}=0}$. Therefore, we obtain that ${R_{23}=\partial_rh_1=0}$.

Put ${p:=G^{-1}\partial_rc}$. The equalities ${G_{23}=\chi=0}$ show that $p$ is orthogonal to $n$ and $e$ with respect to $G$. The vectors $n$ and $m$ form a basis of the $G$-orthogonal complement of $\operatorname{span}_{\mathbb{R}(s)}\{n,e\}$. Consequently, there is an element ${\rho\in L_0}$ such that ${p=\beta\mu^{-1}m+\rho n}$.
Here, the exact expression of the coefficient of $m$ follows from the symmetry identity ${T_{21}=T_{12}}$.

The covectors $Gn,Gm$ are independent. Choose a real rational vector $v=v(s)$ such that ${(Gv)(n)=1}$ and ${(Gv)(m)=0}$.
Put ${f:=c(v)\in L_0}$. We have that $\partial_rf=\rho$. Replacing $c$ by ${c-f\ell}$ and $\lambda$ by $\lambda+f$ preserves ${y=c+\lambda\ell}$ and gives ${G^{-1}\partial_rc=\beta\mu^{-1}m}$. We also replace $h_0$ by ${h_0-fh_1}$ and the horizontal derivations $\partial_s,\partial_r$ by $\partial_s-(\partial_sf)\partial_\lambda,\partial_r-(\partial_rf)\partial_\lambda$ respectively.

We retain the same notation for the new derivations as they agree with the old derivations on ${L}$. Thus, the changes preserve all of ${h=h_0+\lambda h_1,\beta,h_1}$ and ${\partial_rh_1=0}$. Differentiating ${\beta=(\partial_sc)e}$ and using $\langle Gm,\partial_se\rangle=0$ gives $\partial_r\beta=0$. Since $m,{\mu}$ depend only on $s$, we also obtain that $\partial_r^2c=0$.

Put ${a:=R_{13}}$. The identities ${dh=\langle dy,X\rangle}$ and ${G^{-1}\partial_rc=\beta\mu^{-1}m}$ give
$$
\begin{cases}
{a=\partial_sh_1-(\mu/\beta)\partial_rh_0},\\
{R_{22}=\partial_r^2h_0},\\
{\partial_ra=-(\mu/\beta)R_{22}},\\
{G_{22}=\beta^2/\mu}.
\end{cases}
$$
The determinant identities at $t=0$ now give
$$
{R_{22}a^2
=(R_{22}+\beta^2/\mu)(a-r\beta)^2}.
$$
Using the formula for ${\partial_ra}$, we obtain that
$$
{r(2a-r\beta)\partial_ra+(a-r\beta)^2=0}.
$$

Put ${z:=-a/G_{13}}$ and ${K:=\ker(\partial_r\colon L\rightarrow L)}$.
Lemma~\ref{finite_differential_constants}, applied with base field $\mathbb{R}(s)$ and ${t=r}$, shows that $K/\mathbb{R}(s)$ is finite and that $r$ is transcendental over $K$. Since $${L\subseteq M\subseteq\mathbb{R}(x_1,\dots,x_4)},$$ the fields ${L}$ and $K$ are formally real.

The preceding differential equation shows that $r^3(3z^2-3z+1)$ is constant with respect to variable $r$. Thus, there is an element ${c_0\in K}$ such that
$$
{r^3(3z^2-3z+1)=c_0}.
$$
If ${c_0=0}$, then $(6z-3)^2=-3$ in the formally real field ${L}$, which is impossible. Thus, we obtain that ${c_0\neq0}$. Put ${w:=(2a-r\beta)/\beta}$. Direct substitution gives
$$
{r^3+3rw^2=4c_0.}
$$

Put ${E:=K(r,w)}$. Since ${K(r)\subseteq E\subseteq L}$ and ${L/\mathbb{R}(s,r)}$ is finite, the extension ${L/E}$ is finite and separable. The element ${w}$ is nonzero, because $r$ is transcendental over $K$. Differentiating the defining relation gives ${\partial_rw=-(r^2+w^2)/(2rw)\in E}$. Thus, ${L/E}$ is a differential extension.

Finally, the element
$$
{g:=\frac{2}{\beta}\left(r\partial_sh_1-\frac{\mu}{\beta}h_0\right)-\frac{r^2}{2}}
$$
belongs to ${L}$. Direct differentiation gives ${\partial_rg=w}$. This proves the proposition.
\end{proof}

The following differential-algebraic lemma provides the final obstruction.

\begin{lemma}
\label{genus1_nonexactness}
Let $K$ be a field of characteristic zero, let ${c_0\in K}$ be nonzero. Let ${C}$ be {a curve with function field} ${E:=K(r,w)}$ subject to the relation
$$
{r^3+3rw^2=4c_0}.
$$
Then, there is no element ${f\in E}$ such that
$\partial_r f=w$. Here, the derivation $\partial_r$ is trivial on $K$ and satisfies $\partial_rr=1$.
\end{lemma}

\begin{proof}
The valuation at $r=0$ of ${(4c_0-r^3)/(3r)}$ is $-1$. Hence, this element is not a square in $K(r)$, and {the extension} ${E/K(r)}$ is quadratic. It is enough to extend the field of constants to an algebraic closure of $K$. Suppose that ${\partial_rf=w}$ for some ${f\in E}$. Since ${E}$ is quadratic over $K(r)$, there exist $A,B\in K(r)$ such that ${f=A+Bw}$. Differentiating this equality and the defining equation of ${E}$, and then comparing the coefficients of $1$ and ${w}$, gives $\partial_rA=0$ and
$$
{\partial_rB-\frac{r^3+2c_0}{r(4c_0-r^3)}B=1}.
$$
At $r=0$, the coefficient of $B$ in this equation has residue $-1/2$. At every root of ${4c_0-r^3}$, it has residue $1/2$. Comparison of the local orders shows that $B$ has no pole and vanishes at each of these four points. At every other finite point, the derivative of a pole of $B$ would have strictly larger order than the remaining terms. Thus, the function $B$ is a polynomial divisible by ${r(4c_0-r^3)}$, and hence {we have that} $\deg(B)\geq4$.

At infinity, the coefficient of $B$ is $1/r+O(r^{-4})$. If $m:=\deg(B)$, the leading term on the left hand side has degree $m-1$ and coefficient $m+1$ times the leading coefficient of $B$, which cannot be equal to $1$ when $m\geq4$. The contradiction proves the lemma.
\end{proof}

The preceding reduction and the lemma now exclude the remaining developable branch.

\begin{proposition}
\label{developable_exclusion}
In fact, the apex-free developable branch in Case 3 of Proposition~\ref{projective_reduction} actually cannot occur.
\end{proposition}

\begin{proof}
Suppose that this branch occurs. Proposition~\ref{developable_differential_reduction} gives a finite separable differential extension ${L/E}$ and an element ${g\in L}$ such that ${\partial_rg=w}$. Put ${f:=\operatorname{Tr}(g)/[L:E]}$. Since the derivation commutes with the field trace, we have that ${\partial_rf=w}$. This contradicts Lemma~\ref{genus1_nonexactness}.
\end{proof}

\subsection{Completion of the Rank Three Case}

All possibilities for the projective kernel image have now been treated. We are now ready to state our main result for generic rank three.

\begin{theorem}
\label{generic-rank3}
If the generic rank of ${R_\phi}$ is 3, then {the potential} $\phi$ admits a pivot.
\end{theorem}

\begin{proof}
Let $k$ be the dimension of the image of the projective kernel map ${\gamma_\phi}$. Proposition~\ref{projective_reduction} gives $k\leq2$. If $k\leq1$, Corollary~\ref{curve_branch} gives that ${Y_\phi}$ has an apex, and Proposition~\ref{rank3_apex} gives a pivot. Suppose now that $k=2$. If ${Y_\phi}$ has an apex, the same conclusion follows. Otherwise, Proposition~\ref{projective_reduction} gives the apex-free developable branch, which is excluded by Proposition~\ref{developable_exclusion}. Therefore, the potential $\phi$ admits a pivot.
\end{proof}

Combining the low-rank result with the rank-three case gives the main conclusion of this section.

\begin{theorem}
\label{singular_residual_Hessian}
Let $\phi\in\mathbb{R}[x_1,\dots,x_4]$ be a unimodular Hessian potential of index $1$. Suppose that {$R_\phi(x)$} is singular for every $x\in\mathbb{R}^4$. Then, the potential $\phi$ admits a pivot.
\end{theorem}

\begin{proof}
Since the polynomial ${\det(R_\phi(x))}$ vanishes identically, the generic rank of ${R_\phi}$ is at most 3. Proposition~\ref{low_rank} proves the assertion when the generic rank is at most 2, while Theorem~\ref{generic-rank3} proves it when the generic rank is 3.
\end{proof}

The preceding theorem completes the singular residual argument. We now conclude the main text with a degeneracy criterion.

\section{Miscellaneous}

\subsection{The Euler–Jacobi Formula}

In this section, we will use the following notations:

For every unimodular Hessian potential $\phi\in\mathbb{R}[x_1,\dots,x_4]$ of index $1$ and degree $d$, after omitting the affine terms and applying an $\mathbb{R}$-linear change of coordinates, we always assume $\phi_0=\phi_1=0$ and $2\phi_2(x)=\langle Gx,x\rangle$, where {we put} $G:=\operatorname{diag}(-1,1,1,1)$. We also use {$$L_\phi(x):=G^{-1}\operatorname{Hess}(\phi_{\geq3})_x.$$} The Monge–Ampère equation for $\phi$ can be written as ${\det(I_4+L_\phi)=1}$. We denote by ${\sigma_k(L_\phi)}$ the sum of the principal $k$-minors of ${L_\phi}$ and put
$$
{\sigma_{\leq3}(L_\phi):=\sum_{k=1}^3\sigma_k(L_\phi).}
$$
Then, the Monge–Ampère equation for $\phi$ implies that ${\det(L_\phi)=-\sigma_{\leq3}(L_\phi)}$. 

The desired conclusion is a criterion for ${\det(L_\phi)=0}$. 

\begin{definition}\label{lambda_proper}
Let ${\omega:=(\omega_1,\dots,\omega_4)}\in\mathbb{Z}^4$ be a positive integral weight vector, so the weights ${\omega_1,\dots,\omega_4}>0$ are assigned to the complex variables $z_1,\dots,z_4$ respectively. We define a {weighted gauge} ${\rho_\omega}\colon\mathbb{C}^4\rightarrow\mathbb{R}$ via
$$
{\rho_\omega(z):=\max\{1,|z_1|^{1/\omega_1},\dots,|z_4|^{1/\omega_4}\}}
$$
and write ${\deg_\omega}$ for the weighted degree. For polynomials $P_1,\dots,P_4\in\mathbb{C}[z_1,\dots,z_4]$ and a vector ${\lambda:=(\lambda_1,\dots,\lambda_4)}\in\mathbb{Q}^4$ such that ${\lambda_i\leq\deg_\omega(P_i)}$, we say that $P:=(P_1,\dots,P_4)$ is weighted ${\lambda}$-proper at ${\eta}\in\mathbb{C}^4$, if the inequality
\begin{equation}\label{properness}
{\sum_{i=1}^4\frac{|P_i(z)-\eta_i|}{\rho_\omega(z)^{\lambda_i}}\geq c}
\end{equation}
holds outside a compact subset of $\mathbb{C}^4$ for some real constant $c>0$.
\end{definition}

\begin{theorem}\label{Euler–Jacobi_Formula}
Let $P\colon\mathbb{C}^4\rightarrow\mathbb{C}^4$ be a generically finite polynomial map admitting a nonempty finite simple fibre over ${\eta}\in\mathbb{C}^4$. Suppose that $P$ satisfies inequality~(\ref{properness}), and that $Q\in\mathbb{C}[z_1,\dots,z_4]$ satisfies
$$
{\deg_\omega(Q)+\sum_{i=1}^4\omega_i<\sum_{i=1}^4\lambda_i.}
$$
Then, it holds that 
\begin{equation}\label{residue_sum}
{\sum_{P(z)=\eta}\frac{Q(z)}{\det(dP_z)}=0.}
\end{equation}
\end{theorem}

\begin{proof}
Choose a positive $m\in\mathbb{Z}$ such that $m\lambda_i\in\mathbb{Z}$ for every $i\in\{1,\dots,4\}$. After replacing $P(z)$ and $Q(z)$ by $P(z+a)$ and $Q(z+a)$ respectively for a generic $a\in\mathbb{C}^4$, we may assume that ${P^{-1}(\eta)}\subseteq(\mathbb{C}^\times)^4$. This replacement of $z$ by $z+a$ preserves the weighted degrees and inequality~(\ref{properness}), up to a change of the compact set and the constant $c>0$. Put ${\pi_m(\xi):=(\xi_1^{m\omega_1},\dots,\xi_4^{m\omega_4})}$. Since ${\rho_\omega(\pi_m(\xi))=\max\{1,|\xi_1|,\dots,|\xi_4|\}^m}$, the map ${P\circ\pi_m-\eta}$ is ${m\lambda}$-proper. Moreover, we have that
$$
{\deg\big((Q\circ\pi_m)\det(d\pi_m)\big)
\leq m\deg_\omega(Q)+m\sum_{i=1}^4\omega_i-4
\leq m\sum_{i=1}^4\lambda_i-5.}
$$
The generalized Jacobi formula \cite{VidrasYger2001} therefore gives that the pulled-back residue sum $S$
is zero. The map $\pi_m$ is unramified over the fibre and has degree ${\delta:=m^4\omega_1\omega_2\omega_3\omega_4}$. Hence, the residue sum $S$ is this nonzero degree $\delta>0$ times the required sum in equation~(\ref{residue_sum}), which proves the assertion.
\end{proof}

\begin{corollary}\label{corollary_of_Euler–Jacobi}
Let ${\eta}\in\mathbb{C}^4$ be a generic point. If ${G^{-1}\nabla\phi_{\geq3}}$ is weighted ${\lambda}$-proper at ${\eta}$ and 
$$
{\deg_\omega(\sigma_{\leq3}(L_\phi))<\sum_{i=1}^4(\lambda_i-\omega_i),}
$$
then {we have that} ${\det(L_\phi)\equiv0}$.
\end{corollary}

\begin{proof}
Suppose, for a contradiction, that ${\det(L_\phi)\not\equiv0}$. Put ${P:=G^{-1}\nabla\phi_{\geq3}}$, so that {we have} ${dP=L_\phi}$. Then, the genericity of ${\eta}$ implies that $P$ admits a nonempty finite simple fibre over ${\eta}$. Applying Theorem~\ref{Euler–Jacobi_Formula} with ${Q=\sigma_{\leq3}(L_\phi)}$, we obtain that the corresponding residue sum in Theorem~\ref{Euler–Jacobi_Formula} is zero, while the identity ${\sigma_{\leq3}(L_\phi)=-\det(L_\phi)}$ shows that every summand is $-1$, which is a contradiction. Therefore, it must hold that ${\det(L_\phi)\equiv0}$.
\end{proof}

\subsection{The General Quintic Potential}

As a final application of the machinery developed above, we prove the Hessian conjecture in Lorentzian signature for polynomial potentials of degree at most $5$.

We first record two elementary lemmas.

\begin{lemma}\label{odd_weighted_restriction}
Let $\phi\in\mathbb{R}[x_1,\dots,x_4]$ be a unimodular Hessian potential of index $1$, and let ${E}$ be a linear subspace of $\mathbb{R}^4$ of dimension ${\dim_\mathbb{R}(E)=r}$. Assign positive integral weights ${\omega_1,\dots,\omega_r}$ to fixed linear coordinates $x_1,\dots,x_r$ on ${E}$. Let {$\operatorname{in}_{\omega}(\phi|_{E})$} be the highest weight component of ${\phi|_{E}}$, and let {$N$} be its weight. If {$N$} is odd, then {we have that}
$$
{\operatorname{rank}(\operatorname{Hess}(\operatorname{in}_{\omega}(\phi|_{E})))\leq2.}
$$
\end{lemma}

\begin{proof}
Let ${D_\omega(t):=\operatorname{diag}(t^{\omega_1},\dots,t^{\omega_r})}$ be the corresponding weighted dilation operator. For every ${x\in E}$, we have that
\begin{equation}\label{limit}
{t^{-N}D_\omega(t)^\top\operatorname{Hess}(\phi|_{E})_{D_\omega(t)x}D_\omega(t)}
\rightarrow{\operatorname{Hess}(\operatorname{in}_{\omega}(\phi|_{E}))_x}
\end{equation}
as ${t\rightarrow+\infty}$. Since every matrix on the left-hand side of (\ref{limit}) before taking the limit is the restriction of a matrix of index $1$, the matrix ${\operatorname{Hess}(\operatorname{in}_{\omega}(\phi|_{E}))_x}$ has at most one negative eigenvalue.

Now, let ${S_\omega:=\operatorname{diag}((-1)^{\omega_1},\dots,(-1)^{\omega_r})}$. Weighted homogeneity gives
$$
{S_\omega^\top\operatorname{Hess}(\operatorname{in}_{\omega}(\phi|_{E}))_{S_\omega x}S_\omega=-\operatorname{Hess}(\operatorname{in}_{\omega}(\phi|_{E}))_x.}
$$
Applying the first conclusion above at ${S_\omega x}$, we obtain that ${\operatorname{Hess}(\operatorname{in}_{\omega}(\phi|_{E}))_x}$ has at most one positive eigenvalue, and hence its rank is at most two.
\end{proof}

\begin{lemma}\label{binary_quintic_terminal}
Let $F_5,F_3,F_1\in\mathbb{R}[x,y]$ be homogeneous of degrees $5,3,1$, respectively. Suppose that $F_5$ depends essentially on both variables $x$ and $y$. Suppose that
\begin{equation}\label{determinant}
\det(\operatorname{Hess}(F_5+sF_3+s^2F_1/2))=0,
\end{equation}
where the Hessian above is taken in variables $x,y$ and $s$.
Then, it holds that $F_1=0$.
\end{lemma}

\begin{proof}
Suppose, for a contradiction, that $F_1\neq0$. After a linear change of the variables $x,y$, we may then assume that $F_1=x$. The coefficient of $s^3$ in the displayed determinant in (\ref{determinant}) is ${-\partial_y^2 F_3}$. Hence, there exist constants $a,b\in\mathbb{R}$ such that $F_3=ax^3+bx^2y$. So, the coefficient of $s^2$ in the displayed determinant in (\ref{determinant}) is now ${-\partial_y^2 F_5}$. Hence, there exist constants $c,d\in\mathbb{R}$ such that $F_5=cx^5+dx^4y$.
The coefficient of $s$ in the displayed determinant in (\ref{determinant}) is
$$
2bx^4\big((3ab-4d)x+3b^2y\big).
$$
We obtain that $b=0$. The constant coefficient in the displayed determinant in (\ref{determinant}) is then $-16d^2x^7$, and hence we obtain that $d=0$, which contradicts the essential dependence of $F_5$ on both variables. Therefore, we conclude that $F_1=0$.
\end{proof}

\begin{lemma}\label{general_quintic_rank_two}
Let $\phi\in\mathbb{R}[x_1,\dots,x_4]$ be a unimodular Hessian potential of index $1$ and degree $5$. If the essential rank of $\phi_5$ is two, then the potential $\phi$ admits a constant pivot.
\end{lemma}

\begin{proof}
As before, we put {$H_{\phi_j}(x):=\operatorname{Hess}(\phi_j)_x$} for $2\leq j\leq5$. Let ${K_5:=\mathcal{K}(H_{\phi_5})}$ be the common kernel of {$H_{\phi_5}(x)$} for $x\in\mathbb{R}^4$, and choose a complementary space $U$ so that {we have} ${\mathbb{R}^4=U\oplus K_5}$. Write $x=(u,z)$ for the coordinates on the direct sum ${U\oplus K_5}$ so that we may write ${\phi_5=F_5(u)}$, where {the polynomial} ${F_5}$ is an essential binary quintic. Put {$P:=H_{\phi_5}|_{U}$}. We also put ${C:=H_{\phi_4}|_{K_5}}$.

At a generic point, the matrix $P$ is non-singular and has signature $(1,1)$. Put
$$
H(\lambda):={H_{\phi_2}}+\lambda{H_{\phi_3}(x)}+\lambda^2{H_{\phi_4}(x)}+\lambda^3{H_{\phi_5}(x)}={H_\phi(\lambda x)}.
$$
With respect to the direct sum decomposition ${U\oplus K_5}$, write $H(\lambda)$ in its block form
$$
H(\lambda)=
\begin{pmatrix}
{H_{11}(\lambda)}&{H_{12}(\lambda)}\\
{H_{21}(\lambda)}&{H_{22}(\lambda)}
\end{pmatrix}.
$$
For $|\lambda|$ sufficiently large, the matrix {$H_{11}(\lambda)$} is non-singular. The Schur complement of {$H_{11}(\lambda)$} in the ambient $H(\lambda)$ is
$$
{S(\lambda):=H_{22}(\lambda)-H_{21}(\lambda)H_{11}(\lambda)^{-1}H_{12}(\lambda).}
$$
Since {$H_{11}(\lambda)-\lambda^3P$}, {$H_{12}(\lambda)$} and {$H_{21}(\lambda)$} are asymptotically $O(\lambda^2)$, and {since} {$H_{22}(\lambda)=\lambda^2C+O(\lambda)$}, we have that
$$
{S(\lambda)}=\lambda^2C+O(\lambda).
$$
In the block decomposition, the index of $H(\lambda)$ is the sum of the indices of {$H_{11}(\lambda)$} and ${S(\lambda)}$. Since both $H(\lambda)$ and {$H_{11}(\lambda)$} have index $1$, we obtain that ${S(\lambda)}>0$ for $|\lambda|$ sufficiently large. Letting $\lambda\rightarrow\pm\infty$, we obtain that $C\geq0$. The coefficient of $\lambda^{10}$ in $\det(H(\lambda))$ is a nonzero constant multiple of $\det(P)\det(C)$. Consequently, we have that $\det(C)=0$.

Lemma~\ref{quadratic_matrix_dichotomy} shows that either $C$ has a nonzero constant kernel vector or ${C=\rho\,\ell\otimes\ell}$, where {the constant} $\rho>0$ and the two components of ${\ell}$ are independent linear forms. Suppose that the second case occurs. Now, write $z=(z_1,z_2)$ and ${\ell=(\ell_1,\ell_2)}$. For indices $1\leq i,j,k\leq2$, {put $\partial_i:=\partial/\partial z_i$}. The Hessian integrability identities give
$$
{\partial_j(\ell_k\ell_i)=\partial_i(\ell_k\ell_j).}
$$
The independence of ${\ell_1,\ell_2}$ implies that both components are independent of $z$. After a linear change of coordinates, we may therefore assume that ${\ell=u}$. Double integrations with respect to $z$ show that the expression $\phi_4-\rho\langle u,z\rangle^2/2$ is affine in $z$.

Give the variables in $U$ weight two and those in ${K_5}$ weight three. The highest weight component of $\phi$ is
$$
{\Psi(u,z)=F_5(u)+\frac{\rho}{2}\langle u,z\rangle^2,}
$$
as the remaining part of $\phi_4$ and every cubic term have weight at most nine, while every term of degree at most two has weight at most six. Direct computation and Euler's identity then immediately yield
$$
{\det(\operatorname{Hess}(\Psi))=\rho^3\langle u,z\rangle^2(3\rho\langle u,z\rangle^2-20F_5(u))\neq0,}
$$
which contradicts Lemma~\ref{weighted_hessian_face}. Hence, the matrix $C$ has a nonzero constant kernel vector ${\xi\in K_5}$.

To proceed, we put ${F_3:=D_\xi\phi_4}$. The polynomial ${F_3}$ is a binary cubic. We also put ${F_1:=D_\xi^2\phi_3}$.
On the Zariski open subset of $U$ where {the matrix} $P$ is invertible, let ${B}$ denote the coefficient of $\lambda^2$ in {$H_{12}(\lambda)$}, and put
$$
{S_1:=H_{\phi_3}(x)|_{K_5}-B^\top P^{-1}B.}
$$
Since ${H_{22}(\lambda)=\lambda^2C+\lambda H_{\phi_3}(x)|_{K_5}+O(1)}$, the matrix ${S(\lambda)}$ has the expansion
$$
{S(\lambda)=\lambda^2C+\lambda S_1+O(1)}.
$$

Suppose that $C$ has rank one at a generic point. Then, the coefficient of $\lambda^9$ in the determinant of $H(\lambda)$ is a nonzero constant multiple of ${\det(P)\operatorname{tr}(\operatorname{adj}(C)S_1)}$.
Since the kernel of $C$ is ${\mathbb{R}\xi}$, there exists a nonzero polynomial ${\chi}$ such that
$$
{\operatorname{adj}(C)=\chi\, \xi\otimes \xi.}
$$
The vanishing of the coefficient of $\lambda^9$ in the determinant of $H(\lambda)$ gives ${\langle S_1\xi,\xi\rangle=0}$.

Suppose now that $C=0$. Then, the positivity of ${S(\lambda)}$ as $\lambda\rightarrow\pm\infty$ gives ${S_1=0}$. Since ${B\xi=\nabla F_3}$ and ${\langle H_{\phi_3}(x)\xi,\xi\rangle=F_1}$, in either case we obtain that
$$
{F_1=\langle P^{-1}\nabla F_3,\nabla F_3\rangle.}
$$
The right-hand side ${\langle P^{-1}\nabla F_3,\nabla F_3\rangle}$ depends only on $u$, so {we obtain that} ${F_1}$ also depends only on $u$.

Let $s$ be the coordinate in the direction {$\xi$} on ${U\oplus\mathbb{R}\xi}$. Give the variables in $U$ weight one and give $s$ weight two. It is readily seen that
$$
{\operatorname{in}_{\omega}(\phi|_{U\oplus\mathbb{R}\xi})=F_5(u)+sF_3(u)+\frac{s^2}{2}F_1(u),}
$$
and every omitted term has weight at most four.
Lemma~\ref{odd_weighted_restriction} gives that
$$
{\operatorname{rank}(\operatorname{Hess}(\operatorname{in}_{\omega}(\phi|_{U\oplus\mathbb{R}\xi})))\leq2,}
$$
and hence the determinant of this Hessian vanishes. Lemma~\ref{binary_quintic_terminal} now gives ${F_1=0}$. Thus, the second directional derivatives of $\phi_5,\phi_4,\phi_3$ along {$\xi$} all vanish. Since the remaining homogeneous pieces have degree at most two, the polynomial {$D_\xi^2\phi$} is constant. Therefore, the vector {$\xi$} is a constant pivot.
\end{proof}

We next treat the remaining essential-rank-one case.

\begin{lemma}\label{general_quintic_rank_one}
Let $\phi\in\mathbb{R}[x_1,\dots,x_4]$ be a unimodular Hessian potential of index $1$ and degree $5$. If the essential rank of $\phi_5$ is one, then the potential $\phi$ admits a constant pivot.
\end{lemma}

\begin{proof}
{Use the notation} {$H_{\phi_j}(x):=\operatorname{Hess}(\phi_j)_x$} for $2\leq j\leq5$. Choose a leading variable $x$ for $\phi_5$ and some complementary variables $z\in\mathbb{R}^3$ such that $\phi_5=\alpha x^5$ for some $\alpha\in\mathbb{R}^{\times}$. Let ${K_5:=\mathcal{K}(H_{\phi_5})}$ be the {direction space tangent to the $z$-coordinates}, and put ${C:=H_{\phi_4}|_{K_5}}$.

We put
$$
H(\lambda):={H_{\phi_2}}+\lambda{H_{\phi_3}(x,z)}+\lambda^2{H_{\phi_4}(x,z)}+\lambda^3{H_{\phi_5}(x,z)}.
$$
With respect to the direct sum decomposition ${\mathbb{R}\partial_x\oplus K_5}$, write
$$
H(\lambda)=
\begin{pmatrix}
{H_{11}(\lambda)}&{H_{12}(\lambda)}\\
{H_{21}(\lambda)}&{H_{22}(\lambda)}
\end{pmatrix}.
$$
We have that {$H_{11}(\lambda)=20\alpha\lambda^3x^3+O(\lambda^2)$}. It also holds that {$H_{12}(\lambda)=O(\lambda^2)$}, {$H_{21}(\lambda)=O(\lambda^2)$} and {$H_{22}(\lambda)=\lambda^2C+O(\lambda)$}. At each point with $x\neq0$, choose the sign of $\lambda$ for which {$H_{11}(\lambda)<0$} as $|\lambda|\rightarrow\infty$. The Schur complement of {$H_{11}(\lambda)$} in $H(\lambda)$ is
$$
{S(\lambda):=H_{22}(\lambda)-H_{21}(\lambda)H_{11}(\lambda)^{-1}H_{12}(\lambda)=\lambda^2C+O(\lambda).}
$$
Since {$H_{11}(\lambda)<0$}, it accounts for the unique negative eigenvalue of $H(\lambda)$, and hence {we obtain that} ${S(\lambda)}>0$. After division by $\lambda^2$ and passage to the limit, we obtain that $C\geq0$. Continuity then gives the same conclusion when $x=0$. The coefficient of $\lambda^9$ in $\det(H(\lambda))$ is $20\alpha x^3\det(C)$. Consequently, we have that $\det(C)=0$.

Let $r\leq2$ be the maximum rank of the matrix-valued function $C=C(x,z)$. Lemma~\ref{convex_cylinder} now gives a nonzero constant vector in the common kernel of the matrices $C(1,z)$ for $z\in {K_5}$, so that, up to an affine term, the polynomial $\phi_4(1,-)$ descends to the quotient. Repeating the argument gives a fixed linear subspace ${K_C\subseteq K_5}$ of dimension at least $3-r$ in the common kernel of $C(1,z)$ for $z\in {K_5}$. The generic rank of $C$ forces equality, namely ${\dim_\mathbb{R}(K_C)=3-r}$. Homogeneity also gives $C(x,z)=x^2C(1,z/x)$ when $x\neq0$. Thus, continuity shows that ${K_C}$ is contained in the common kernel of $C$. Therefore, there exists a constant vector ${\xi\in K_C}$ such that, for every ${v\in K_C}$, we have that
$$
{D_v\phi_4=\langle \xi,v\rangle x^3.}
$$

We put ${Q_1:=H_{\phi_3}|_{K_C}}$. Since ${\partial_x D_v\phi_4=3x^2\langle \xi,v\rangle}$ for every ${v\in K_C}$, the coefficient of the linear term $\lambda$ in ${S(\lambda)|_{K_C}}$ is
$$
{S_1:=Q_1-\frac{9x}{20\alpha}\xi\otimes\xi.}
$$
Fix $x=1$ and choose the sign of $\lambda$ for which {$H_{11}(\lambda)<0$}. Then, for every ${v\in K_C}$, the polynomial ${\langle S_1(1,z)v,v\rangle}$ is affine in $z$ and has a fixed sign. It is therefore constant in $z$. Polarization shows that every coefficient of $z$ in ${S_1}$ vanishes. By homogeneity, there exists a constant symmetric matrix $Q_0$ on ${K_C}$ such that ${Q_1=xQ_0}$.

Give $x$ weight one and the variables ${u\in K_C}$ weight two. The highest weight component of the restriction of $\phi$ to ${\mathbb{R}x\oplus K_C}$ is
$$
{\Theta(x,u):=\alpha x^5+x^3\langle \xi,u\rangle+\frac{x}{2}\langle Q_0u,u\rangle,}
$$
and every omitted term has weight at most four. Lemma~\ref{odd_weighted_restriction} gives
$$
\operatorname{rank}(\operatorname{Hess}(\Theta))\leq2.
$$

Suppose first that $r=0$. Then, we have that ${\dim_\mathbb{R}(K_C)=3}$. Since $xQ_0$ is a principal block of $\operatorname{Hess}(\Theta)$, the matrix $Q_0$ has a nonzero kernel vector, which is a constant pivot for $\phi$.

Suppose next that $r=1$. Then, we have that ${\dim_\mathbb{R}(K_C)=2}$. A direct computation gives
$$
{\det(\operatorname{Hess}(\Theta))=x^5\big(20\alpha\det(Q_0)-9\langle\operatorname{adj}(Q_0)\xi,\xi\rangle\big)-x\det(Q_0)\langle Q_0u,u\rangle=0.}
$$
Comparison of the terms that are quadratic in $u$ gives $\det(Q_0)=0$. A nonzero vector in the kernel of $Q_0$ is again a constant pivot.

Finally, suppose that $r=2$. Then, we have that ${K_C=\mathbb{R}\partial_s}$ for a transverse coordinate $s$. Write
${\partial_s^2\phi_3=\mu x}$
for some constant $\mu\in\mathbb{R}$. If $\mu=0$, the direction $\partial_s$ is a constant pivot for $\phi$.

From now on, we assume, for a contradiction, that $\mu\neq0$. Choose complementary variables $y\in\mathbb{R}^2$. Integration in $s$ gives $\phi_4=bx^3s+F_4(x,y)$ and 
$$
\phi_3=\frac{\mu}{2}xs^2+sB_2(x,y)+G_3(x,y),
$$
where {the polynomial} $B_2$ is homogeneous quadratic. The full Hessian $\operatorname{Hess}(\phi)$ hence forms an affine pencil
$$
{\operatorname{Hess}(\phi)=M_0+sM_1,}
$$
where the matrix ${M_0}$ is independent of the variable $s$ and ${M_1:=\operatorname{Hess}(bx^3+\mu xs+B_2)}$.
Write
$$
B_2(x,y)=\beta x^2+x\langle p,y\rangle+\frac{1}{2}\langle Ey,y\rangle
$$
for a constant $\beta\in\mathbb{R}$, a constant vector $p\in\mathbb{R}^2$ and a constant symmetric matrix $E$. In the ordered coordinate system $(x,s,y_1,y_2)$, we have that
$$
{M_1}=
\begin{pmatrix}
6bx+2\beta&\mu&{p^\top}\\
\mu&0&0\\
p&0&E
\end{pmatrix}.
$$
For each fixed $(x,y)$, Lemma~\ref{lorentz_flat_pencil} gives ${\operatorname{rank}(M_1)\leq2}$. The upper left $2$ by $2$ block of ${M_1}$ has determinant $-\mu^2$, so its rank is two. The inverse of this block has zero upper left entry, so its Schur complement in ${M_1}$ is exactly $E$. Since ${\operatorname{rank}(M_1)\leq2}$, we obtain that $E=0$. The matrix ${M_1}$ has the fixed kernel
$$
{K_M:=\left\{\left(0,-\langle p,v\rangle/\mu,v\right):v\in\mathbb{R}^2\right\}.}
$$
Fix the basis of ${K_M}$ that projects to the standard basis of the $y_1y_2$-plane, so that in the resulting basis of ${\operatorname{span}\{\partial_x,\partial_s\}\oplus K_M}$, the coefficient of $s^2$ in ${\det(M_0+sM_1)}$ is ${-\mu^2\det(M_0|_{K_M})}$.
It follows that ${\det(M_0|_{K_M})=0}$. Every vector in ${K_M}$ has vanishing $x$-component, and hence {we have that} ${H_{\phi_5}|_{K_M}=0}$. The projection from ${K_M}$ to the $y$-space is an isomorphism, and the homogeneous component of degree two of ${M_0|_{K_M}}$ is the Hessian matrix of $F_4(x,-)$ in the two variables $y_1$ and $y_2$, whose determinant is nonzero because $C$ has rank two at a generic point and common kernel $\mathbb{R}\partial_s$. This is a contradiction. We therefore conclude that $\mu=0$, so the direction $\partial_s$ is a constant pivot for $\phi$.
\end{proof}

\begin{proposition}\label{quintic_pivot}
Let $\phi\in\mathbb{R}[x_1,\dots,x_4]$ be a unimodular Hessian potential of index $1$. Suppose that $\deg(\phi)=5$. Then, the potential $\phi$ admits a constant pivot.
\end{proposition}

\begin{proof}
Lemma~\ref{lorentz_parity} shows that the essential rank of $\phi_5$ is one or two. The assertion follows from Lemma~\ref{general_quintic_rank_one} and Lemma~\ref{general_quintic_rank_two}.
\end{proof}

\begin{corollary}\label{degree_five_hessian_conjecture}
Let $\phi\in\mathbb{R}[x_1,\dots,x_4]$ be a unimodular Hessian potential of index $1$ and degree at most $5$. Then, the gradient mapping $\nabla\phi\colon\mathbb{R}^4\rightarrow\mathbb{R}^4$ is a polynomial automorphism.
\end{corollary}

\begin{proof}
Put $d:=\deg(\phi)$. If $d\leq3$, the Hessian differences depend linearly on the four variables, so the complex dimension of $\mathcal{H}_\phi$ is at most four. Corollary~\ref{low_dimensional_Hesse_system} then proves the assertion. If $d=4$, Lemma~\ref{quartic_pivot} provides a constant pivot. If $d=5$, Proposition~\ref{quintic_pivot} provides a constant pivot. Proposition~\ref{constant_pivot_inversion} completes the remaining two cases.
\end{proof}

\section{Conclusion}

This article studies the four-dimensional Hessian conjecture under the assumption that the Hessian has Lorentzian signature. The arguments are organised around the constant-pivot condition. Once a constant pivot is available, the causal decomposition reduces polynomial inversion either to the known three-dimensional case or to the lightlike normal form, while the timelike case is rigid. Consequently, the gradient mapping of every potential admitting a constant pivot is a polynomial automorphism. The appendix strengthens this point by showing that, after an affine linear change of coordinates, every such potential has the null normal form
$$
\phi(t,x,y,z)=xt+F(x,y,z).
$$

The highest-weight method detects a pivot from a non-singular weighted face. The degree-separation argument proves the two-layer theorem and its multi-layer extension in the stated range $d\geq4k+3$. The Hesse system gives a complementary complex-algebraic approach: the existence of a complex pivot is equivalent to the rank bound $\operatorname{rank}(\mu_\phi)\leq55$, and the complex-to-real principle converts this into a real Lorentzian pivot criterion. In particular, the gradient mapping is a polynomial automorphism whenever the Hesse system has complex dimension at most four. If every member of the Hesse system is singular, a pivot again exists; in the remaining base-point-free case of dimension at most six, the nonlinear part separates into two ternary summands.

A degree-free conclusion is obtained when $R_\phi(x)$ is singular for every $x\in\mathbb{R}^4$. The lower-rank cases follow from the small-rank classification. In generic rank three, the canonical kernel field and its projectivisation reduce the problem to the geometry of the gradient-image hypersurface. The apex branch produces a pivot directly. In the remaining branch, the projective kernel image is a developable surface; rational parametrisation and the induced differential-field structure reduce this possibility to a non-exactness statement on a genus-one curve, which gives a contradiction. Hence, every potential with singular residual Hessian admits a constant pivot.

Finally, the weighted restriction argument and the analysis of the two possible essential ranks of the leading quintic form prove the conjecture for all potentials of degree at most five. The general four-dimensional Lorentzian case remains open. By the results above, any counterexample in this setting would have degree at least six, a generically non-singular residual Hessian and a Hesse system of complex dimension at least five. It would also have to lie outside the two-layer and multi-layer degree-separation classes treated here. Thus, the results isolate the remaining case while supplying several algebraic, geometric and differential criteria by which a constant pivot, and hence polynomial invertibility, can be detected.

\appendix
\section{A Null Normal Form}

Using the methods developed above, we observe that the existence of a pivot in fact implies the existence of an affine translation symmetry.

\begin{theorem}
Let $\phi\in\mathbb{R}[x_1,\dots,x_4]$ be a unimodular Hessian potential of index $1$. Then, the potential $\phi$ admits a constant pivot if and only if
$$
{\phi(t,x,y,z)=xt+F(x,y,z)}
$$
for a ternary polynomial $F$ and some coordinate system ${(t,x,y,z)}$ on $\mathbb{R}^4$, with ${t,x,y,z}$ all being affine in {the original coordinate system $(x_1,\dots,x_4)$}.
\end{theorem}

\begin{proof}
We first prove the following claim: 

For $Q,b\in\mathbb{R}[x,y,z]$, we put ${H_Q:=\operatorname{Hess}(Q)}$ and ${H_b:=\operatorname{Hess}(b)}$. We also assume that
$$
{\det(H_Q+\lambda H_b)=c_0\in\mathbb{R}^{\times}}
$$
and that ${H_Q+\lambda H_b}$ has index $1$ at every point $(x,y,z)\in\mathbb{R}^3$, for every $\lambda\in\mathbb{R}$. Then, we claim that there exists a nonzero constant vector $e\in\mathbb{R}^3$ such that $H_b e=0$ and $\langle H_Q e,e\rangle=0$.

Indeed, the argument of Lemma~\ref{lorentz_flat_pencil} already gives ${\operatorname{rank}(H_b)\leq2}$. Suppose first that ${H_b}$ has rank $2$ at a generic point in $\mathbb{R}^3$. The classification theorem of ternary singular Hessians proved in \cite{deBondtEssen2004} and \cite{deBondt2015} gives, after a linear change of coordinates and addition of an affine polynomial, either $b=b(x,y)$ or
$$
b=a(x)+f(x)y+g(x)z.
$$

In the first case, investigating the coefficient of $\lambda^2$ gives
$$
{\bigl(\partial_x^2 b\,\partial_y^2 b-(\partial_x\partial_y b)^2\bigr)\partial_z^2 Q=0,}
$$
so we may take $e=\partial_z$. 

In the second case, after interchanging $y$ and $z$ if necessary, we may assume that $f'\neq0$ and we put
$$
{\begin{cases}
p:=g'/f'\in\mathbb{R}(x),\\
q:=y+pz,\\
e:=\partial_z-p\partial_y.
\end{cases}}
$$
Then, investigating the coefficient of $\lambda^2$ gives $D_e^2Q=0$. Hence, over the function field $\mathbb{R}(x)$, we may write
$$
{Q=U_0(q)+zU_1(q)}
$$
for some ${U_0,U_1\in\mathbb{R}(x)[q]}$.
In the basis $(\partial_x,\partial_y,e)$, put
$$
{h:=a''+f''q+p'f'z.}
$$
It is readily seen that the coefficient of $\lambda$ is
$$
{(\partial_q U_1)\bigl(2f'(\partial_x U_1+p'\partial_q U_0+2p'z\partial_q U_1)-(\partial_q U_1)h\bigr)=0.}
$$
If ${\partial_q U_1\neq0}$, comparison of the coefficients of $z$ gives ${3p'f'\partial_q U_1=0}$, and hence $p'=0$. If ${\partial_q U_1=0}$, then  ${U_1=U_1(x)}$ and
$$
{-(\partial_q^2 U_0)(U_1'+p'\partial_q U_0)^2=\det(H_Q)=c_0.}
$$
Since ${c_0\neq0}$, we have that the degree $m$ of ${\partial_q U_0}$ in the variable $q$ is at least $m\geq1$. If $p'\neq0$, then the left hand side has degree $3m-1>0$ in $q$, which is impossible. Thus, we again obtain that $p'=0$. Therefore, the vector $e$ is constant. Moreover, we have ${H_b e=0}$ and ${\langle H_Q e,e\rangle=0}$.

Suppose now that ${H_b}$ has rank $1$ at a generic point in $\mathbb{R}^3$. The same classification gives, after a constant linear change of coordinates, that $b=b(x)$ up to an affine polynomial. Investigating the coefficient of $\lambda$ gives
$$
{\partial_y^2 Q\,\partial_z^2 Q-(\partial_y\partial_z Q)^2=0.}
$$
The classification theorem of singular Hessians in \cite{deBondtEssen2004} and \cite{deBondt2015} over $\mathbb{R}(x)$ gives, after interchanging $y$ and $z$ if necessary, that
$$
{Q=U_0(x,y+p(x)z)+zU_1(x).}
$$
For $q:=y+pz$ and $e:=\partial_z-p\partial_y$, direct expansion gives
$$
{\det(H_Q)=-(\partial_q^2 U_0)(U_1'+p'\partial_q U_0)^2=c_0.}
$$
The same degree comparison gives $p'=0$. Hence, the constant vector $e$ has the required properties. Finally, if ${H_b=0}$, de Bondt's theorem \cite{deBondt2015} gives a constant null direction for ${H_Q}$ when $Q$ is nonquadratic, while the quadratic case follows from the Lorentzian signature. 

This proves the required claim.

Now, let {$\xi$} be a constant pivot for $\phi$, and put ${a:=D_\xi^2\phi}$. If ${a<0}$, Lemma~\ref{timelike_pivot_rigidity} shows that $\phi$ is quadratic, and its constant Hessian of signature $(+,+,+,-)$ trivially has a nonzero null vector. If ${a=0}$, the vector {$\xi$} itself is already a null pivot.

It remains to assume that ${a>0}$. After a unimodular linear change of coordinates, write {$\xi:=\partial_t$} and
$$
{\phi(u,t)=\frac{a}{2}t^2+b(u)t+c(u).}
$$
We also put
$$
{Q:=c-\frac{b^2}{2a}.}
$$
For the coordinate ${s:=a t+b(u)}$, the Schur-complement identities give the pencil
$$
{H(s):=\operatorname{Hess}(Q)+\frac{s}{a}\operatorname{Hess}(b),}
$$
where each $H(s)$ has index $1$ and
$$
{\det(H(s))=-1/a.}
$$
The claim proved above now provides a nonzero constant vector $e\in\mathbb{R}^3$ such that $\operatorname{Hess}(b)e=0$ and $D_e^2Q=0$,
which implies that $D_eb\in\mathbb{R}$ is a constant. Write
$$
{\hat{\xi}:=e-\frac{D_eb}{a}\partial_t.}
$$
Then, a direct computation gives that the double directional derivative of $\phi$ along $\hat{\xi}$ is $D_e^2Q=0$.

After a linear change of coordinates, we may put ${\hat{\xi}:=\partial_t}$ and write
$$
\phi(t,u)=A(u)t+B(u)
$$
for some polynomials $A$ and $B$.
The proof of Lemma~\ref{lightlike_pivot_normal_form} shows that $A=f(s)$ for a linear coordinate $s$, where $f'$ is a nonzero constant. Put ${x:=A=f(s)}$, which is an affine linear coordinate. Extending ${x}$ to an affine linear coordinate system ${(t,x,y,z)}$, we finally obtain
$$
{\phi(t,x,y,z)=xt+F(x,y,z)}
$$
for some ${F\in\mathbb{R}[x,y,z]}$.

Conversely, the normal form ${\phi(t,x,y,z)=xt+F(x,y,z)}$ satisfies ${\partial_t^2\phi=0}$, so {the vector} $\partial_t$ is a constant pivot for $\phi$.
\end{proof}

\section*{Acknowledgements}

The author would like to express his sincere gratitude to Zihuan Feng and Yang Zhang (Institut de mathématiques, École polytechnique fédérale de Lausanne) for their suggestive ideas and very careful proofreading of the algebraic geometry statements in this article. The author also thanks Chen Wang (Johannes Kepler Universität Linz) for pointing out an elegant use of Bertini's theorem, which leads to a substantial improvement of Section~5.2. Special thanks should go to Cheng He (Ningbo University) as well for inspiring discussions about the PDE aspects of the Hessian conjecture. 

The author is indebted to the Russian mathematical community, especially to all the professors teaching the "Math in Moscow" program, and most importantly, to the author's supervisor Alexander Petrovich Veselov at Loughborough University, as they cultivated the author's mathematical literacy and maturity.

This research was completed while the author was studying at the Mathematics Institute of the University of Warwick.
The author would therefore like to thank the University of Warwick for its hospitality.

\section*{Statements and Declarations}

\noindent\textbf{Funding.} No funding was received to assist with the preparation of this manuscript.

\noindent\textbf{Competing interests.} The author declares no relevant financial or non-financial interests.

\noindent\textbf{Data availability.} Data sharing is not applicable to this article.

\noindent\textbf{Use of generative AI.} During the preparation of this manuscript, the author used OpenAI’s
ChatGPT to help handle minor details and grammatical issues.

\end{document}